\documentclass{amsart}

\usepackage{amsmath}
\usepackage{amssymb}
\usepackage{amsthm}
\usepackage[foot]{amsaddr}
\usepackage{dsfont}
\usepackage{hyperref}
\usepackage{enumerate}
\usepackage{enumitem}
\usepackage{multicol}

\hypersetup{colorlinks=true}
\theoremstyle{plain}
\newtheorem{proposition}{Proposition}
\newtheorem{theorem}{Theorem}
\newtheorem{lemma}{Lemma}
\newtheorem{corollary}{Corollary}
\newtheorem{claim}{Claim}
\newtheorem*{theorem*}{Theorem}

\theoremstyle{definition}
\newtheorem{definition}{Definition}
\newtheorem*{fact}{Fact}

\theoremstyle{remark}
\newtheorem{remark}{Remark}

\newtheorem{question}{Question}

\title{Approximate Lattices in Higher-Rank Semi-Simple Groups}
\author{Simon Machado}
\address{Institute for Advanced Study \\
              1 Einstein Drive \\
              Princeton 08540 \\
              New Jersey \\
              USA} 
 \email{machado@ias.edu}
\newcommand{\R}{\mathbb{R}}

\newcommand{\N}{\mathbb{N}}

\newcommand{\Q}{\mathbb{Q}}
\newcommand{\G}{\mathbb{G}}

\def\L{\Lambda}

\DeclareMathOperator{\GL}{GL}
\DeclareMathOperator{\SL}{SL}
\DeclareMathOperator{\id}{id}
\DeclareMathOperator{\Ad}{Ad}

\DeclareMathOperator{\Tr}{Tr}

\DeclareMathOperator{\PVS}{PVS}

\begin{document}
 \begin{abstract}
  We show that strong approximate lattices in higher-rank semi-simple algebraic groups are arithmetic.  
 \end{abstract}

 \maketitle
 \tableofcontents
 \section{Introduction}
 
 An $l$-approximate subgroup, in the sense of Tao \cite{MR2501249}, is a symmetric subset $X$ of a group $G$ that contains the identity and that is closed under multiplication up to an error contained in a finite subset $F$ of size at most $l$ i.e. such that the set of products $X^2 :=\{x_1x_2 |  x_1, x_2 \in X \}$ is contained in $FX:=\{fx | f \in F, x \in X\}$. The theory of expander graphs and of group growth \cite{MR3309986}, additive combinatorics \cite{MR2289012} and model theory \cite{MR2833482, van2013approximate} motivated the study of finite approximate subgroups,  eventually leading to the structure theorem proved in \cite{MR3090256} (see also, \cite{MR3237440}). More recently, investigations of certain types of infinite approximate subgroups of locally compact groups have gained attention, see e.g. \cite{MR3529116, MR3361775, MR3606726},  \cite{bjorklund2016approximate, MR3915195}, \cite{MR3438951, jing2021nonabelian}.

In particular, a very natural family of examples of infinite approximate groups arises already in the early work of Yves Meyer \cite{meyer1972algebraic} on harmonious sets and has since become a key concept in the theory of aperiodic order and mathematical quasi-crystals \cite{moody1997meyer, MR1340198, MR2253164}.  Aperiodic order is concerned with the study of Delone sets in Euclidean spaces, namely relatively dense and uniformly discrete subsets of $\mathbb{R}^n$.  A subset $X \subset \mathbb{R}^n$  is  \emph{uniformly discrete} if there is $r>0$ such that the distance between any two distinct points of $X$ is bounded below by $r$, and \emph{relatively dense} if there is $R \geq 0$ such that any point of $\mathbb{R}^n$ is within distance $R$ of some point of $X$. We note, in particular, that de Bruijn’s \cite{deBruijn1981algebraicI,deBruijn1981algebraicII} showed that the set of vertices of the Penrose rhombus tiling (P3) is a Delone set that is also an approximate subgroup of $\mathbb{R}^2$.

Meyer proved in \cite{meyer1972algebraic} that a Delone set $\Lambda \subset \mathbb{R}^n$ that is also an approximate subgroup must be commensurable to a \emph{cut-and-project set} (also known as \emph{model set}), i.e. a set of the form $\pi_{\mathbb{R}^n}\left( \Gamma \cap \left(\mathbb{R}^n \times W\right) \right)$, where $\Gamma$ is a lattice in $\mathbb{R}^n \times \mathbb{R}^m$, $\pi_{\mathbb{R}^n}: \mathbb{R}^n \times \mathbb{R}^m \rightarrow \mathbb{R}^n$ is the projection to the first factor and $W$ is a compact neighbourhood of the origin in $\mathbb{R}^m$. Here, two subsets $\Lambda_1, \Lambda_2$ of a group $G$ are said to be \emph{commensurable} if there is a finite subset $F$ of $G$ such that $\Lambda_i$ is contained in $F \Lambda_j \cap \Lambda_j F$ for $i \neq j \in \{1,2\}$.  Meyer also characterized such Delone sets in Fourier-theoretic terms as harmonious sets (see \cite{moody1997meyer, schreiber1973approximations, MR1400744}).  

There are many striking consequences of Meyer's theorem, see \cite{moody1997meyer}. One of them is a kind of sum-product phenomenon - proved ten years before the seminal paper by Erd\H{o}s and Szemer\'{e}di \cite{MR820223}. It asserts that:

  \begin{theorem*}[{\cite[Thm VI and X]{meyer1972algebraic}}]
  If $\Lambda\subset \mathbb{R}$ is a Delone approximate subgroup of $(\mathbb{R},+)$ and is stable under product ($\Lambda\Lambda \subset \Lambda$), then there is a number field $K$ such that $\Lambda$ is contained in and commensurable to the set $\PVS(K)$ of primitive \emph{Pisot--Vijayaraghavan--Salem numbers} of $K$.
  \end{theorem*}   
  
  \noindent The Pisot--Vijayaraghavan--Salem numbers are those algebraic integers $\alpha$ such that every Galois conjugate of $\alpha$ is either $\alpha$ or has modulus $\leq 1$. We denote by $\PVS(K)$ the set of those Pisot--Vijayaraghavan--Salem numbers in $K$ that are primitive i.e.  of degree $[K:\mathbb{Q}]$. It is easy to see that they form a Delone set of $\mathbb{R}$ for every real number field $K$. Furthermore, $\PVS(K) \cup \{0\}$ is an approximate subgroup of $\mathbb{R}$ (see Subsection \ref{Subsection: Definition and first properties first part}).

The sum-product phenomenon for finite fields \cite{MR2053599} was used in \cite{MR2415382} to establish a structure theorem for approximate subgroups of $\SL_2(\mathbb{F}_p)$. Similarly, Meyer's sum-product theorem above hints at the possibility of a structure theory for uniformly discrete approximate subgroups in non-commutative Lie groups. In this vein, it was observed by Bj\"{o}rklund and Hartnick \cite{bjorklund2016approximate} that the cut-and-project construction above generalises very naturally to arbitrary locally compact groups: if $\mathbb{R}^n$ and $\mathbb{R}^m$ are replaced by locally compact groups $G$ and $H$ respectively, $W$ is a compact symmetric neighbourhood of the identity in $H$, and $\Gamma$ is a lattice in the product $G \times H$, then the \emph{model set} $\Lambda:= \pi_G\left(\Gamma \cap \left( G \times W\right)\right)$  is a uniformly discrete approximate subgroup in $G$. 

Taking $\Gamma$ co-compact in the above construction forces $\Lambda$ to be Delone. More generally, the above construction makes sense if $\Gamma$ is not assumed co-compact, but only of finite co-volume. This has been investigated in \cite{bjorklund2016approximate} and it has been shown in \cite{bjorklund2016aperiodic} that for a model set $\Lambda$ defined as above with the help of a lattice $\Gamma$ (of finite co-volume, but not necessarily co-compact), the closed $G$-orbit

$$\Omega_\Lambda := \overline{ \left\{ g\Lambda, g \in G\right\}},$$

\noindent of $\Lambda$ in the Chabauty space of closed subsets of $G$ (see Section \ref{Section: Star-approximate lattices}), carries a finite G-invariant Borel measure $\mu_\Lambda$. The corresponding measure space $(\Omega_\Lambda, \mu_\Lambda)$ is called the invariant hull of $\Lambda$ in $G$. The existence of an invariant hull of finite volume for an arbitrary Delone set  $\Lambda$ in the classical setting of Euclidean spaces is a natural condition studied for instance in \cite{MR1992666, MR1798991, MR2308136}.

 Following \cite{bjorklund2016approximate} we will say that a uniformly discrete approximate subgroup of a locally compact group is a \emph{strong approximate lattice} if its invariant hull admits a finite $G$-invariant Borel measure. Note that a subgroup of $G$ is a strong approximate lattice if and only if it is a lattice in $G$. This naturally suggests the following:
 
 \begin{question}[Bj\"{o}rklund--Hartnick]
 Let $\Lambda$ be a strong approximate lattice in a locally compact group $G$. Is $\Lambda$ commensurable to a model set arising as above from a locally compact group $H$ and a lattice $\Gamma \subset G \times H$?
 \end{question}

We answer this question affirmatively for semi-simple algebraic groups whose simple factors have higher-rank. This complements previous work by the author that dealt with nilpotent Lie groups \cite{machado2020approximate}, soluble Lie groups \cite{machado2019infinite} and amenable locally compact groups \cite{machado2019goodmodels}. Note moreover that in the companion paper \cite[Thm 1.10]{machado2019goodmodels} we also prove that every strong approximate lattice in a connected Lie group can be decomposed into an amenable part and a semi-simple part, extending classical theorems of Bieberbach and Auslander in the setting of discrete subgroups.

 Our approach in this paper takes advantage of rigidity and arithmeticity results available in higher rank. Our method thus yields a result that is simultaneously reminiscent of Meyer's sum-product theorem and a generalisation of Margulis' arithmeticity theorem \cite{MR1090825}.

\begin{theorem}\label{First theorem}
Let $\Lambda$ be a strong approximate lattice in the $\mathbb{R}$-points $\mathbb{G}(\mathbb{R})$ of an absolutely almost simple algebraic group $\mathbb{G}$ defined over $\mathbb{R}$ with $\mathbb{R}$-rank $\geq 2$. There are a number field $K \subset \mathbb{R}$, a finite set of places $S$ of $K$, and an absolutely almost simple group $\mathbb{G}'$ defined over $K$ with $\mathbb{G}'(\mathbb{R})= \mathbb{G}(\mathbb{R})$, such that 
$$ \Lambda \text{ is commensurable to } \mathbb{G}'(\PVS(K)) \text{ and } \langle \Lambda \rangle  \text{ is commensurable to } \mathbb{G}'(\mathcal{O}_{K,S})$$
where $\mathcal{O}_{K,S}$ denotes the ring of $S$-integers of $K$ and $\langle\Lambda\rangle$ denotes the group generated by $\Lambda$. 
\end{theorem}

Recall that to define the subgroup $\mathbb{G}'(\mathcal{O}_{K,S})$ we implicitly choose a $K$-group embedding $\pi:\mathbb{G}'(K) \rightarrow \SL_n(K)$ for some integer $n \geq 1$ and we write $\mathbb{G}'(\mathcal{O}_{K,S})=\pi^{-1}(\SL_n(\mathcal{O}_{K,S}))$ \cite[I.3.1]{MR1090825}. Likewise, we define the approximate subgroup $\mathbb{G}'(\PVS(K))$ as the inverse image via $\pi$ of the following subset of matrices of $\SL_n(K)$: 
\begin{equation}
 \{ A \in \SL_n(K)| I-A \text{ and } I-A^{-1} \text{ have entries in } \PVS(K) \cup \{0\}\} . \label{Eq: Pisot}
 \end{equation}
Then $\mathbb{G}'(\PVS(K))$ is always strong approximate lattice (Lemma \ref{Lemma: Pisot implies strong}). Furthermore, replacing $\PVS(K)$ in (\ref{Eq: Pisot}) with any subset of $K$ commensurable to $\PVS(K)$ would simply yield another approximate subgroup commensurable with $\mathbb{G}'(\PVS(K))$. This shows, for instance, that we could very well consider a similar construction based on the Pisot--Vijayaraghavan numbers of $K$ - those algebraic integers whose Galois conjugates have absolute value $<1$ - rather than $\PVS(K)$ without changing the content of Theorem \ref{First theorem}.

Note moreover, that if $\Lambda$ is $l$-approximate and not commensurable to a lattice, then using the recent non-abelian Brunn--Minkowski inequality proved by Jing, Tran and Zhang \cite{jing2021nonabelian} we can show: 
 \begin{itemize}
 \item  $\dim \mathbb{G} \left(\vert \{ \text{ Arch. places $v$  s.t. } \mathbb{G}'(K_v) \text{ non-compact }\}\vert -1 \right) \leq 9\log^2(l);$
 \item  if, moreover, $\Lambda$ is not relatively dense, then $\dim  \mathbb{G} \left([K:\mathbb{Q}] -1 \right) \leq 18\log^2(l)$.
 \end{itemize}
In both formulae, $\log$ denotes the base $2$ logarithm.

We now state the main theorem of this paper. Theorem \ref{Main theorem} below is concerned with a larger class of subsets than the strong approximate lattices, the class of $\star$-approximate lattices defined in Section \ref{Section: Star-approximate lattices}. This includes sets of the form $X^{-1}X$, where $X$ is a subset of $G$ with a $G$-invariant hull of finite volume and such that $(X^{-1}X)^3$ is uniformly discrete. It turns out that under these assumptions alone, the set $X^{-1}X$ is an approximate subgroup: 

\begin{fact}[Corollary \ref{Corollary: Approximate subgroup condition is shallow}]
Let $X \in \mathcal{C}(G)$ be such that $(X^{-1}X)^3$ is uniformly discrete and $\Omega_X$ admits a finite $G$-invariant measure. Then $X^{-1}X$ is an approximate subgroup.
\end{fact}

We prove: 

\begin{theorem}\label{Main theorem}
 Consider a finite set $A$, characteristic $0$ local fields $(k_{\alpha})_{\alpha \in A}$ and almost $k_{\alpha}$-simple algebraic groups $(\mathbb{G}_{\alpha})_{\alpha \in A}$ defined over $k_{\alpha}$ of $k_{\alpha}$-rank $\geq 2$. Let $\Lambda$ be a $\star$-approximate lattice in $G_A:=\prod_{\alpha \in A}\mathbb{G}_{\alpha}(k_{\alpha})$. Then there are $B$ finite, characteristic $0$ local fields $(k_{\beta})_{\beta \in B}$,  absolutely almost simple algebraic groups $(\mathbb{H}_{\beta})_{\beta\in B }$ defined over $k_{\beta}$ and a lattice $\Gamma \subset G_A \times \prod_{\beta \in B}\mathbb{H}_{\beta}(k_{\beta}) $ whose intersection with $\prod_{\beta \in B}\mathbb{H}_{\beta}(k_{\beta}) $ is trivial and such that 
 $$\langle \Lambda \rangle = p_A(\Gamma) \text{ and } \Lambda \text{ is contained in and commensurable to } p_A\left(\Gamma \cap  G_A \times  \prod_{\beta \in B} U_{\beta}\right) $$
 where $p_A $ is the natural projection to $G_A$ and the $U_{\beta}$'s  are symmetric compact neighbourhoods of the identity in $\mathbb{H}_{\beta}(k_{\beta})$.
\end{theorem}

Theorem \ref{Main theorem} states that both $\Lambda$ and the group $\langle \Lambda \rangle$ it generates have an arithmetic origin. Combined with Margulis' arithmeticity theorem this yields the number theoretic information needed to show Theorem \ref{First theorem}. 

Note that the local fields $k_{\beta}$ may be Archimedean or non-Archimedean. However, when  $\Lambda$ is $l$-approximate, a result of Carolino \cite{MR3438951} in combination with Theorem \ref{Main theorem} ensures that one can find $\Lambda' \subset \Lambda^4$ such that $C:=C(l)$ translates of $\Lambda'$ cover $\Lambda$, both $\Lambda'$ and $\langle \Lambda' \rangle$ are commensurated by the elements of $\langle \Lambda \rangle$ and the conclusions of Theorem \ref{Main theorem} hold for $\Lambda'$ with all local fields $k_{\beta}$ Archimedean.

Note moreover that  all fields considered in this article will be of characteristic $0$. In fact, in positive characteristic every strong approximate lattice is a lattice, see \cite{Hrushovski2020quasimodels}. It also turns out that a theorem similar to Theorem \ref{Main theorem} holds as well in rank one, or when factors of rank one are allowed. This is shown independently by Hrushovski in \cite{Hrushovski2020quasimodels} using model-theoretic tools. Our ergodic method however requires the higher rank assumption on each factor.

Approximate lattices in $G_A$ provide examples of tilings in the symmetric space, for example by considering Voronoi cells around a $\Lambda$-orbit, which can sometimes be shown to be aperiodic. Although it is not directly related to our work, we note that in \cite{MR1452434} Mozes constructed a finite set of tiles with matching rules that can tile the symmetric space $G_A/K$ only aperiodically. His Theorem 2 also required the higher rank assumption (this time because of its reliance on quasi-isometric rigidity).

 Theorem \ref{Main theorem} is proved following Margulis' original strategy for arithmeticity: we prove first a superrigidity result and then a finite generation result (see \cite{MR1090825}). To derive these two results we use Zimmer's superrigidity for cocycles \cite{zimmer2013ergodic} (we use, in fact, a version due to Fisher and Margulis \cite{MR2039990} that does not require connectedness of the algebraic hull of the cocycle considered). As in the case of lattices, given a strong approximate lattice $\Lambda$ and a Borel section $s: \Omega_{\Lambda} \rightarrow G$  such that $s(X) \in X$ we define the natural cocycle $\alpha_s$ on $\Omega_{\Lambda}$ by $\alpha_s(g,X) := s(gX)^{-1} g s(X)$. Note that $\alpha(\cdot,X)$ takes values in $X^{-1}X$ and, hence, in $\Lambda^{-1}\Lambda \subset \langle \Lambda \rangle$. 
 
 A crucial ingredient in Zimmer's proof of Margulis' superrigidity is a result of Mackey \cite{MR44536} describing the structure of cocycles arising from a transitive action. This fails in our situation. Indeed, Mackey's proof crucially uses the existence of large stabilisers, whereas $\Omega_{\Lambda}$ typically has trivial stabilisers. Circumventing this difficulty is the central point of our work. Given a group homomorphism $T: \langle\Lambda\rangle \rightarrow H$ with target a topological group we first relate the range of the cocycle $T \circ \alpha_s$ to $T(\Lambda)$ as follows. 

\begin{theorem}\label{Range cocycle}
Let $\Lambda$, $G_A$ be as in Theorem \ref{Main theorem}. Let $T$ and $\alpha$ be as above. Suppose that $T \circ \alpha_s$ is cohomologous to a cocycle taking values in a closed subgroup $L \subset H$. Then there is $h \in H$ such that for all neighbourhoods of the identity $\mathcal{V} \subset H$ there is $\Lambda' \subset \Lambda^2$ commensurable to $\Lambda$ such that 
$$T(\Lambda') \subset \mathcal{V}hLh^{-1}\mathcal{V}^{-1}$$
\end{theorem}

Theorem \ref{Range cocycle} gives more than enough information when applied to group homomorphisms taking values in the group of isometries of a CAT(0) space (Propositions \ref{Proposition: Consequence two-sided error in CAT(0) space} and \ref{Proposition: Constant cocycles in semi-simple groups}). In particular, this puts us in a position to apply cocycle superrigidity (e.g. \cite{zimmer2013ergodic}) and obtain the following:

\begin{theorem}\label{Superrigidity}
Let $\Lambda$, $G_A$ be as in Theorem \ref{Main theorem}. Take a local field $k$ and a simple (centreless) algebraic group $\mathbb{L}$ defined over $k$. Let $T:\langle\Lambda\rangle \rightarrow \mathbb{L}(k)$ be a group homomorphism. Then one of the following is true: 
\begin{enumerate}
\item there is a continuous group homomorphism $\pi: G_A \rightarrow \mathbb{L}(k)$ extending $T$; 
\item $T(\Lambda)$ is bounded.
\end{enumerate}
\end{theorem}
\noindent This result directly extends Margulis' superrigidity from lattices to $\star$-approximate lattices. We prove in fact a special case first: Theorem \ref{Range cocycle} together with cocycle superrigidity yield the case when $\dim \mathbb{L} \leq \dim \mathbb{G}_{\alpha}$. We use it in combination with a finiteness statement (Proposition \ref{Proposition: Rigidity of compact approximate subgroups}) which is itself a consequence of one of the main results of our companion paper \cite[Thm. 4.1]{machado2019goodmodels}. This special case is central in the proof of  Theorem \ref{Main theorem}, which in turn yields the general case of Theorem \ref{Superrigidity}.

If we assume instead that $T$ is a unitary representation of $\langle\Lambda\rangle$ on a Hilbert space $\mathcal{H}$, we can introduce a unitary representation of $G$ on $L^2(\Omega_{\Lambda}, \mathcal{H})$ by
$$ \pi(g)( f)(X) := T\circ \alpha_s(g^{-1},X)^{-1}\left(f(g^{-1}X)\right),$$ 
for all $g$ and almost all $X$. Theorem \ref{Range cocycle} enables us to transfer information from the induced representation of $G$ to the representation of $\Lambda$: if $\pi$ fixes a unit vector, then there is a unit vector $x$ such that $T(\Lambda)(x)$ is totally bounded (Proposition \ref{Proposition: Heredity of property (T)}). 

We build upon this remark to define a notion of \emph{property (T) for approximate subgroups}:  an approximate subgroup $\Lambda$ has \emph{property (T)} if for any unitary representation $\pi$ of $\langle\Lambda\rangle$ almost with invariant vectors there is a sub-representation $\sigma$ with $\sigma(\Lambda)$ totally bounded in the strong topology (Definition \ref{Definition: Property T for approximate subgroups}). Property (T) for approximate subgroups implies the relative Kazhdan property of the pair $(\langle\Lambda\rangle, \Lambda)$ as defined by Cornulier \cite{MR2245534} and reduces to the usual property (T) when $\Lambda$ is a subgroup. Moreover, a $\star$-approximate lattice has property (T) if and only if the ambient group has property (T). This fact is key and leads to the following:

\begin{theorem}\label{Theorem: Approximate lattices in property T groups are finitely generated}
Let $\Lambda$ be a $\star$-approximate lattice in a locally compact group with property (T). Then $\langle\Lambda \rangle$ is finitely generated. 
\end{theorem}

\noindent Theorem \ref{Theorem: Approximate lattices in property T groups are finitely generated} is instrumental in the proof of Theorem \ref{Main theorem}. As in the proof of Margulis' arithmeticity \cite{MR1090825} and following an idea of Vinberg \cite{MR0279206}, we use it to prove that the Galois automorphisms of the field of traces of $\Ad \Lambda$ give rise to group homomorphisms between simple groups over local fields, to which we can apply the above-mentioned superrigidity result for target groups of dimension at most $\dim \mathbb{G}_{\alpha}$ (Theorem \ref{Superrigidity}).

 \section{$\star$-Approximate Lattices}\label{Section: Star-approximate lattices}
 \subsection{Definition and First Properties}\label{Subsection: Definition and first properties first part}
 For any two subsets $X$ and $Y$ of a group $G$ we define $XY:=\{xy |x \in X,y \in Y\}$ and $X^n := \{x_1\cdots x_n | x_1,\ldots,x_n \in X\}$ for all non-negative integers $n$. We write $\langle X \rangle$ the subgroup generated by $X$. We say that two subsets $X,Y$ of a group $G$ are \emph{commensurable} if there is a finite subset $F$ of $G$ such that $X\subset FY$ and $Y \subset FX$.  A subset $A$ of a group $G$ is a ($l$-)\emph{approximate subgroup} for some positive integer $l$ if (i) the identity element $e$ belongs to $A$, (ii) the subset $A$ satisfies $A=A^{-1}$ and (iii) there is a finite subset $F \subset G$ with $|F| \leq l$ such that $A^2\subset FA$. 
 
 A subset $X$ of a locally compact group $G$ is \emph{locally finite} if for all compact subsets $K \subset G$ we have $|X \cap K| < \infty$. The subset $X$ is said \emph{(left) uniformly discrete} if $e$ is isolated in $X^{-1}X$. Equivalently, $X$ is uniformly discrete if there exists a neighbourhood of the identity $V$ with $|gV \cap X| \leq 1$ for all $g \in G$.
 
 Equip now the space $\mathcal{C}(G)$ of closed subsets of $G$ with the \emph{Chabauty--Fell topology}. The Chabauty--Fell topology on the space $\mathcal{C}(G)$ of closed subsets of $G$ is defined as the topology generated by the subsets
 $$ U_K:=\{ C \in \mathcal{C}(G) | C \cap K =\emptyset \} \text{ and } U^V:=\{ C \in \mathcal{C}(G) | C \cap V \neq\emptyset \} $$
 for all $K \subset G$ compact and $V \subset G$ open. When $G$ is second countable, the Chabauty-Fell topology makes $\mathcal{C}(G)$ into a compact metrizable $G$-space with the obvious action $(g,X) \mapsto gX$. To any $X_0 \in \mathcal{C}(G)$ one can therefore associate a compact metrizable $G$-space $\Omega_{X_0}$ called the \emph{(left)-invariant hull} and defined as the closure of the orbit of $X_0$. A \emph{strong approximate lattice} is thus defined as a uniformly discrete approximate subgroup $\L$ of $G$ such that the $G$-space $\Omega_{\L}$ admits a $G$-invariant Borel probability measure that is  \emph{proper} i.e. $\nu(\{\emptyset\})=0$. See \cite{bjorklund2016approximate} for this and more.

 Note moreover that a sequence $(X_n)_{n \geq 0}$ of elements of $\mathcal{C}(G)$ converges to $X$ if and only if $(1)$ for all $x \in X$ there are $x_n \in X_n$ for all integers $n \geq 0$ with $x_n \rightarrow x$; $(2)$ for all increasing sequences of positive integers $(i_k)_{k \in \N}$ and any sequence $(x_{i_k})_{k \in \N}$ with $x_{i_k} \rightarrow x$, we have $x \in X$ (\cite{MR139135},\cite[Prop. 1.8]{MR2406240}).

 In what follows, we will often work with almost simple algebraic groups defined over a field $k$ and we refer to the preliminary section of \cite{MR1090825} (and references therein) for background. In this article we will always consider that almost simple algebraic groups are both connected and linear. When the word ``almost" is omitted, we mean that the group is centreless. Nevertheless, to avoid any ambiguity, we will sometimes add ``centreless".
 
 \subsubsection{PVS numbers}
 
 We will rephrase now the definition of PVS numbers using valuations following a treatment due to Meyer (see \cite[\S II.13.3]{meyer1972algebraic}). This will allow us to quickly show that the PVS numbers provide examples of strong approximate lattices. This point of view will become handy later on in the proof of Theorem \ref{First theorem}.
 
 Let $K \subset \mathbb{R}$ be a number field and let $v_0$ be the Archimedean place on $K$ inherited from the inclusion into $\mathbb{R}$. Let $\mathcal{O}_K$ denote the ring of algebraic integers of $K$ and $S_{\infty}$ denote the set of all Archimedean places of $K$. Write $S=S_{\infty}\setminus \{v_0\}$. For all $v \in S_{\infty}$, let $| \cdot |_v$ be an absolute value associated with $v$. Then the set of Pisot numbers $\PVS(K)$ adjoined with $0$ is equal to the set
 $$\{ \alpha \in \mathcal{O}_K | \forall v \in S, |\alpha|_v \leq 1\}.$$
 It would appear at first glance that the condition that $\alpha$ be primitive has been omitted here. However, when $\alpha \in \mathcal{O}_K$ is not primitive, there is always $v \in S$ such that $|\alpha|_v > 1$. In short, if $\alpha$ is not primitive, then there exists a non-identity field embedding $\iota:K \rightarrow \mathbb{C}$ such that $\iota(\alpha)=\alpha$. The Galois action on the set of valuations then provides $v \in S_{\infty}$ different from $v_0$ such that $|\alpha|_v=|\alpha|_{v_0}>1$.
 
 This point of view works equally well when considering sets of matrices. Define for any $v \in S$ the subset $U_v$ of $\SL_n(K_v)$ as 
 $$ \{ g \in \SL_n(K_v) | g -I, g^{-1} - I \text{ have entries in } O_v \}$$
 where $O_v$ is the unit ball of $| \cdot |_v$ in the completion $K_v$. Then for any $K$-subgroup $\G \subset \SL_n$, the subgroup $\G(\PVS(K))$ is simply $\G(\mathcal{O}_K) \cap \bigcap_{v \in S} p_v^{-1}(U_v)$ where $p_v: \SL_n(K) \rightarrow \SL_n(K_v)$ denotes the natural map. It was proved in \cite[\S II.13.3]{meyer1972algebraic} that the set $\PVS(K) \cup \{0\}$ is a Meyer set in $\mathbb{R}$. We prove here: 
 
 \begin{lemma}\label{Lemma: Pisot implies strong}
 Let $\G \subset \SL_n$ be a $K$-subgroup that is either reductive or unipotent. Then the subset $\G(\PVS(K))$ is a strong approximate lattice in $\G(\mathbb{R})$.
 \end{lemma}
 
 \begin{proof}
 By the Borel--Harish-Chandra theorem (see e.g. \cite[Thm I.3.2.7]{MR1090825}), the diagonal embedding of $\G(\mathcal{O}_K)$ into $\prod_{v \in S_{\infty}} \G(K_v)$ makes $\G(\mathcal{O}_K)$ into a lattice. Now, $\G(\PVS(K)) \subset \G(K_{v_0})=\G(\mathbb{R})$ is equal to 
 $$p_{v_0}\left(\G(\mathcal{O}_K) \cap \left( \G(K_{v_0}) \times \prod_{v\in S} U_{v} \cap \G(K_v)\right)\right).$$ Thus, we are in a situation to apply \cite[Thm. 1.2]{bjorklund2016aperiodic}, which implies that $\G(\PVS(K))$ is indeed a strong approximate lattice. 
 \end{proof}
 
 While we will be interested in the semi-simple case only, we have included unipotent groups in Lemma \ref{Lemma: Pisot implies strong} as it provides a proof that $\PVS(K) \cup \{0\}$ is a Meyer set. Indeed, this reduces to the case where $\G$ is the additive group $\mathbb{G}_a$. 
 
 \begin{remark}
 Note, finally, that the above definition of $\G(\PVS(K))$ depends on an embedding of $\G$ in some $\SL_n$. This is very similar to the definition of the subgroup $\G(\mathcal{O}_{K,S})$. It is well-known that, while it is not possible to provide an intrinsic definition of $\G(\mathcal{O}_{K,S})$ (i.e. independent of the chosen embedding) from the $K$-structure only, all definitions provide commensurable subgroups of $\G(K)$, see \cite[\S I.3.2]{MR1090825} for a more in-depth discussion and proofs. Following a similar path, one could show that any two definitions of $\G(\PVS(K))$ provide commensurable approximate subgroups. Since we will not use this fact, we do not provide details here. 
 \end{remark}

\subsubsection{Extended hull} We define now a generalisation of the invariant hull. 

 \begin{definition}\label{Definition: Extended invariant hull}
  Let $X_0$ be a closed subset of a locally compact second countable group $G$. The \emph{extended invariant hull} $\Omega_{X_0,G}^{ext}$ of $X_0$ in $G$ is the subset of $\mathcal{C}(G)$ defined by
  $$ \{ X \in \mathcal{C}(G) | X^{-1}X \subset \overline{X_0^{-1}X_0} \}.$$
  When the group $G$ considered is clear from context we will simply write $\Omega_{X_0}^{ext}$.
 \end{definition}
 
 The extended invariant hull is a natural generalisation of the invariant hull defined above:
 
 \begin{lemma}
  Let $X_0$ be a closed subset of a locally compact second countable group $G$. Then $\Omega_{X_0,G}^{ext}$ is a closed subset of $\mathcal{C}(G)$ stable under the $G$-action. Thus, the set $\Omega_{X_0,G}^{ext}$ is a metrizable compact continuous $G$-space. Moreover, the invariant hull $\Omega_{X_0}$ is contained in $\Omega_{X_0,G}^{ext}$. 
 \end{lemma}
 
 \begin{proof}
  Stability under the $G$-action is clear. Now, consider a sequence $(X_n)_{n\geq 0}$ of elements of $\Omega_{X_0,G}^{ext}$ that converges towards $X \in \mathcal{C}(G)$ and take $x, y \in X$. There are $x_n, y_n \in X_n$ for all integers $n \geq 0$ such that $x_n \rightarrow x$ and $y_n \rightarrow y$. Note that for all integers $n \geq 0$ we have $x_n^{-1}y_n \in \overline{X_0^{-1}X_0}$. So we have $x^{-1}y = \lim_{n \geq 0} x_n^{-1}y_n \in \overline{X_0^{-1}X_0}$. Hence, the closed subset $X$ belongs to $\Omega_{X_0,G}^{ext}$. Finally, we have the $X_0 \in \Omega_{X_0,G}^{ext}$ so $\Omega_{X_0} \subset \Omega_{X_0,G}^{ext}$.
 \end{proof}

  Questions of continuity and measurability of maps between hulls will be central. We record here a simple result that enables us to prove that some maps are Borel. 
  
   \begin{lemma}\label{Lemma: Criterion Borel measurability 2}
   Let $X$ be a topological space, let $G$ be a locally compact second countable group and let $\Phi : X \rightarrow \mathcal{C}(G)$ be a map. If $\Phi^{-1}(U_K)$ is a Borel subset for every compact subset $K \subset G$, then $\Phi$ is Borel measurable. 
  \end{lemma}
  
  \begin{proof}
   Let $V \subset G$ be any open subset. Then $V$ is $\sigma$-compact since $G$ is second countable locally compact. So let $(K_n)_{n \geq 0}$ be any sequence of compact subsets such that $K_0 \subset K_1 \subset \ldots$ and $\bigcup_{n \geq 0} K_n =V$. Then one checks that $U^V = \bigcup_{n\geq 0} G\setminus U_{K_n}$ so $\Phi^{-1}(U^V)$ is Borel. We conclude that $\Phi$ is Borel.    
   \end{proof}
   
   \subsubsection{$\star$-Approximate lattices}  
  Given a closed subset $X_0$ of some locally compact group we want to study the ergodic theory of $\Omega_{X_0}^{ext}$. Note that not every invariant measure on $\Omega_{X_0}^{ext}$ is interesting however, as the measure $\delta_{\emptyset}$ is always a well defined $G$-invariant Borel probability measure on $\Omega_{X_0}^{ext}$. This observation leads to the following:

 \begin{definition}\label{Definition: Non-trivial measures}
  We say that a left-invariant Borel probability measure $\nu$ on $\Omega_{X_0,G}^{ext}$ is \emph{proper} if $\nu(\{\emptyset\})=0$.
 \end{definition}

 \begin{definition}\label{Definition: star-approximate lattice}
 Let $\L$ be a subset of a locally compact second countable group $G$. We say that $\L$ is a $\star$-\emph{approximate lattice} if:
 \begin{enumerate}
  \item$\L$ is an approximate subgroup;
  \item $\L$ is uniformly discrete;
  \item there is a proper $G$-invariant Borel probability measure $\nu$ on $\Omega_{\L,G}^{ext}$.
 \end{enumerate}
 \end{definition}

  We refer the interested reader to Subsubsection \ref{A remark concerning star-approximate lattices, strong approximate lattices and commensurability} for comments and motivations regarding Definition \ref{Definition: star-approximate lattice}.
 
  Note here again that condition (1) of Definition \ref{Definition: star-approximate lattice} is quite natural and follows from (3) and a slight strengthening of (2) according Corollary \ref{Corollary: Approximate subgroup condition is shallow} from Subsection \ref{Subsection: The Periodization Map}. In addition, proper $G$-invariant Borel probability measures on the extended invariant hull can be related to proper $G$-invariant Borel probability measures on the invariant hull (in the sense of \cite{bjorklund2016approximate}) of some subset.  

 \begin{proposition}\label{Proposition: Equivalent definitions of star-approximate lattices}
  Let $\L$ be a uniformly discrete approximate subgroup of a locally compact second countable group $G$. The following are equivalent: 
  \begin{enumerate}
   \item the approximate subgroup $\L$ is a $\star$-\emph{approximate lattice};
   \item there is $X_0 \subset \L^2$ with $X_0^{-1}X_0 \subset \L^2$ such that there is a proper $G$-invariant ergodic Borel probability measure $\nu$ on $\Omega_{X_0}$ the invariant hull of $X_0$ in $G$.
  \end{enumerate}
 \end{proposition}

 \begin{proof}
   Let us start by showing $(1) \Rightarrow (2)$. Since $\L$ is a $\star$-approximate lattice there exists a proper $G$-invariant ergodic Borel probability measure on the metrizable compact $G$-space $\Omega_{\L, G}^{ext}$. Let $S$ denote the support of $\nu$. Since $\nu$ is ergodic there is an element $X \in S$ whose orbit is dense in $S$. We thus have $S=\overline{G\cdot X}=\Omega_X$. But $X \neq \emptyset$ since $\nu$ is proper. So $\nu$ gives rise to a proper $G$-invariant Borel probability measure on $\Omega_X$. Define now $X_0:=x^{-1}X$ for some $x \in X$. We have $X_0=x^{-1}X \subset \L^2$, $X_0^{-1}X_0=X^{-1}X \subset \L^2$ and $\Omega_X = \Omega_{X_0}$ admits a proper $G$-invariant ergodic Borel probability measure.
   
   Conversely, the inclusion $i:\Omega_{X_0} \rightarrow \Omega_{\L, G}^{ext}$ is injective, continuous and satisfies $i^{-1}(\{\emptyset\}) \subset \{\emptyset\}$. So $i^*\nu$ is a proper $G$-invariant ergodic Borel probability measure on $\Omega_{\L, G}^{ext}$. 
 \end{proof}
 
 Proposition \ref{Proposition: Equivalent definitions of star-approximate lattices} will enable us to utilise tools developed for strong approximate lattices (from e.g. \cite{bjorklund2016approximate,bjorklund2016aperiodic,machado2019goodmodels}). As a first and immediate consequence of Proposition \ref{Proposition: Equivalent definitions of star-approximate lattices} (and \cite[Cor. 2.7]{bjorklund2019borel}) we have the implications:
  \begin{center}
   strong approximate lattice $\Rightarrow$ $\star$-approximate lattice $\Rightarrow$ approximate lattice in the sense of Bj\"{o}rklund--Hartnick \cite[Def. 4.12]{bjorklund2016approximate}. 
  \end{center}
 We moreover have a straightforward corollary of a result from \cite{bjorklund2016approximate}: 

 \begin{corollary}\label{Corollary: Star-approximate lattices are bi-syndetic}
  Let $\L$ be a $\star$-approximate lattice in a locally compact second countable group $G$ and $V \subset G$ be any neighbourhood of the identity. Then there is a compact subset $K \subset G$ such that $V \L K = G$. 
 \end{corollary}

 \begin{proof}
  Apply \cite[Cor. 4.22]{bjorklund2016approximate} to the closed subset $X_0$ given by Proposition \ref{Proposition: Equivalent definitions of star-approximate lattices}. 
 \end{proof}

 The following additional consequence of Proposition \ref{Proposition: Equivalent definitions of star-approximate lattices} and results from \cite{machado2019goodmodels} will be key. 
 
 \begin{corollary}\label{Corollary: Intersection star-approximate lattice and closed subgroup}
  Let $\L$ be a $\star$-approximate lattice in a locally compact second countable group $G$. Let $H \subset G$ be a closed normal subgroup such that the projection $p_{G/H}(\L)$ of $\L$ to $G/H$ is locally finite. Then $p_{G/H}(\L)$ and $H \cap \L^2$ are $\star$-approximate lattices in $G/H$ and $H$ respectively. 
 \end{corollary}

 \begin{proof}
  The fact that $H \cap \L^2$ is a $\star$-approximate lattice is a consequence of Proposition \ref{Proposition: Equivalent definitions of star-approximate lattices} and the proof of \cite[Prop. 6.1]{machado2019goodmodels}. In addition, since $p_{G/H}(\L)$ is locally finite, the map 
  \begin{align*}
   P_{G/H}:\Omega_{\L,G}^{ext} & \longrightarrow \Omega_{p_{G/H}(\L),G/H}^{ext} \\
    X & \longmapsto p_{G/H}(X)
  \end{align*}
 is well-defined and easily seen to be Borel (e.g. by Lemma \ref{Lemma: Criterion Borel measurability 2}). One thus readily checks that the push-forward of any proper $G$-invariant Borel probability measure on  $\Omega_{\L,G}^{ext}$ is a proper $G/H$-invariant Borel probability measure on  $\Omega_{\L,G/H}^{ext}$. 
 \end{proof}

\subsubsection{A remark concerning $\star$-approximate lattices, strong approximate lattices and commensurability}\label{A remark concerning star-approximate lattices, strong approximate lattices and commensurability}

  We believe that there are substantially more $\star$-approximate lattices than strong approximate lattices. In practice, this is illustrated by the fact that intersections of $\star$-approximate lattices with certain closed normal subgroups are also $\star$-approximate lattices (Corollary \ref{Corollary: Intersection star-approximate lattice and closed subgroup}). This particular feature will have a direct impact in the proof of Theorem \ref{Main theorem} as it will enable us to reduce to the case of irreducible $\star$-approximate lattices (Lemma \ref{Lemma: Reduction of star-approximate lattices}).

Note that such stability properties are not at all obvious and, in general, showing that there exists a proper invariant probability measure on the (extended) hull of a given uniformly discrete approximate subgroup is difficult. Indeed, this problem reduces to the challenge of understanding invariant measures on the Chabauty space. Let us briefly mention some open questions along those lines. 

Suppose that $\Lambda_0$ is a strong approximate lattice in a locally compact group $G$. 
\begin{itemize}
\item[(1)] If $\Lambda_1$ is an approximate subgroup commensurable to $\Lambda_0$, is $\Lambda_1$ a strong approximate lattice? 
\item[(2)] If $\Lambda_1 \subset \Lambda_0 \subset \Lambda_2$ are two approximate subgroups commensurable to $\Lambda_0$, must $\Lambda_1$ be a strong approximate lattice? what about $\Lambda_2$?
\end{itemize}

It is the guess of the author that answers to both questions should be negative - even when $\Lambda_0$ is a lattice in a simple Lie group.

The notion of $\star$-approximate lattices partially solves this issue and makes sure that, in the set-up of (2), $\Lambda_2$ is then also a $\star$-approximate lattice. It is our guess, however, that the answer to question (1) should be negative as well for $\star$-approximate lattices. Here are two further questions that we believe are more reasonable. Let $\Lambda_0$ be a $\star$ approximate lattice in a locally compact group $G$:

\begin{itemize}
\item[(3)] let $\Lambda_1 $ be an approximate subgroup commensurable to $\Lambda_0$, is $\Lambda_1^m$ a $\star$-approximate lattice for some $m \geq 0$?
\item[(4)] If so, is there an absolute $m$ (i.e. depending only the ambient group $G$ for instance)?
\end{itemize}
When $\Lambda_0$ is commensurable to a model set, we can answer affirmatively question (3) (an elementary proof follows from ideas developed in \cite[\S 3]{machado2019goodmodels} but would be too long to be included in this remark). In particular, if $\Lambda_0$ is a $\star$-approximate lattice in an ambient group that is a semi-simple algebraic group with higher-rank factors, then Theorem \ref{Main theorem} asserts that it is commensurable to a model set. Thus we can solve question (3) when the ambient group is a semi-simple algebraic group with higher-rank factors. However, even in that case question (4) remains open. 
 \subsection{The Periodization Map}\label{Subsection: The Periodization Map}
 Given a locally compact group $G$ and a closed subgroup $H$ one may use the quotient map $I$ that sends a continuous function $f: G \rightarrow \R$ with compact support to the function $I(f): G/H \rightarrow \R$ defined by $I(f)(gH) := \int_H f(gh)d\mu_H(h)$ to naturally relate the Haar measure on $G$ and the Haar measure on $G/H$ (see e.g. \cite[Section I]{raghunathan1972discrete}). Here as well, we can define a map called the \emph{periodization} map, inspired by \cite{bjorklund2016approximate, 10.2307/1969027}, that will allow us to exploit the strong connections between measures on the extended hull $\Omega_{X_0}^{ext}$ and measures on $G$. 
 
 \begin{definition}\label{Definition: Periodization map}
  Let $X_0$ be a uniformly discrete subset of a locally compact second countable group $G$. Then the \emph{periodization} map defined by 
  \begin{align*}
   \mathcal{P}_{X_0}: \   C^0_c(G) & \longrightarrow C^0(\Omega_{X_0}^{ext}) \\
                 f       &\longmapsto \left( X \mapsto \sum_{x \in X} f(x) \right)
  \end{align*}
 where $C^0_c(G)$ is the set of continuous functions on $G$ with compact support and $C^0(\Omega_{X_0}^{ext})$ is the set of continuous functions on $\Omega_{X_0}^{ext}$ is well-defined, continuous and left-equivariant. 
 \end{definition}

 \subsubsection{Measures on the hull vs. measures on the group} Our first crucial observation is that the periodization maps show how to relate the measure of basic open subsets of the Chabauty-Fell topology in the extended invariant hull with the measure of open and compact subsets of $G$ through a simple formula:
 
 \begin{lemma}\label{Lemma: Measure of pull-backs of small neighbourhoods through the periodization map}
  Let $X_0$ be a uniformly discrete subset of a locally compact second countable group $G$ and let $\nu$ be a Borel probability measure on $\Omega_{X_0}^{ext}$. Take an open subset $V \subset G$ and a compact subset $K \subset G$. Then 
  $$ \left(\mathcal{P}_{X_0}\right)^{*}\nu(V) = \int_{\Omega_{X_0}^{ext}} |X \cap V| d\nu(X) \leq |V^{-1}V\cap X_0^{-1}X_0|\nu(U^V),$$
  and 
  $$ \left(\mathcal{P}_{X_0}\right)^{*}\nu(K) = \int_{\Omega_{X_0}^{ext}} |X \cap K| d\nu(X) \leq |K^{-1}K\cap X_0^{-1}X_0|(1 - \nu(U_K)).$$
 \end{lemma}
 
 \begin{proof}
  Choose a sequence of non-negative continuous functions with compact support $(\phi_n)_{n\geq 0}$ that point-wise converges increasingly to $\mathds{1}_V$. For all integers $n \geq 0$ and all $X \in \Omega_{X_0}$ we have
  $$\mathcal{P}_{X_0}(f_n)(X) = \sum_{x \in X} \phi_n(x) = \sum_{x \in X \cap V} \phi_n(x).$$ 
  We have 
  $$ \lim_{n \rightarrow  \infty}\mathcal{P}_{X_0}(\phi_n)(X) =\sum_{x \in X \cap V} \mathds{1}_V(x) = |X \cap V|.$$
  In addition, the sequence $(\mathcal{P}_{X_0}(\phi_n))_{n\geq 0}$ is increasing. Hence, the map $X \mapsto |X\cap V|$ is measurable and, according to the monotone convergence theorem, we have 
  $$ \left(\mathcal{P}_{X_0}\right)^{*}\nu(V) = \int_{\Omega_{X_0}^{ext}} |X \cap V| d\nu(X).$$
  Similarly, choose a sequence of non-negative continuous functions with compact support $(\psi_n)_{n\geq 0}$ that point-wise converges decreasingly to $\mathds{1}_K$. Again 
  $$ \lim_{n \rightarrow  \infty}\mathcal{P}_{X_0}(\psi_n)(X) =\sum_{x \in X \cap K} \mathds{1}_K(x) = |X \cap K|.$$
  So as above 
  $$ \left(\mathcal{P}_{X_0}\right)^{*}\nu(K) = \int_{\Omega_{X_0}^{ext}} |X \cap K| d\nu(X).$$
  
  Finally, the right-hand inequalities follow from the facts that for any $X \in \Omega_{X_0}^{ext}$ and any subset $Y \subset G$ we have 
  $$ |X \cap Y| \leq |X^{-1}X \cap Y^{-1}Y | \leq |X_0^{-1}X_0 \cap Y^{-1}Y|.$$
  \end{proof}
  
  This formula is of particular importance when applied to a proper $G$-invariant Borel probability measure, as, in this situation, the pull-back by the periodization map is always a Haar measure (see also \cite{bjorklund2016approximate, 10.2307/1969027} for similar facts in other frameworks):
  
  \begin{corollary}\label{Corollary: Pull-back through periodization of G-invariant measure is a Haar-measure}
   Let $X_0$ be a uniformly discrete subset of a locally compact second countable group $G$ and let $\nu$ be a proper $G$-invariant Borel probability measure on $\Omega_{X_0}^{ext}$. Then $(\mathcal{P}_{X_0})^*\nu$ is a Haar-measure on $G$. 
  \end{corollary}

  As a first application of the periodization we show a generalisation of the well-known fact that groups containing a lattice are unimodular (see also \cite{bjorklund2016approximate} for a version for strong approximate lattices):
  
  \begin{proposition}\label{Proposition: Envelope of a star-approximate lattice is unimodular}
 Let $G$ be a locally compact second countable group and suppose that $G$ contains a $\star$-approximate lattice $\L$. Then $G$ is unimodular.  
 \end{proposition}

 \begin{proof}
  Take a proper $G$-invariant Borel probability measure $\nu$ on $\Omega_{\L}^{ext}$ and set $\mu:=(\mathcal{P}_{\L})^*\nu$. We know by Corollary \ref{Corollary: Pull-back through periodization of G-invariant measure is a Haar-measure} that $\mu$ is a left-Haar-measure on $G$ and let $\Delta_G$ denote the modular function of $G$. Choose a neighbourhood of the identity $V \subset G$ such that $V^{-1}V \cap \L^4 = \{e\}$. Then for all $\lambda \in \L$ we have 
  $$ \Delta_G(\lambda)\mu(V) = \mu(V\lambda) = (\mathcal{P}_{\L})^*\nu(V\lambda) \leq 1 $$
  according to Lemma \ref{Lemma: Measure of pull-backs of small neighbourhoods through the periodization map}. We thus find $\Delta_G(\lambda) \leq \mu(V)^{-1}$. But $\Delta_G: G \rightarrow \R_{>0}$ is a continuous group homomorphism. So $\Delta_G$ has bounded image according to Corollary \ref{Corollary: Star-approximate lattices are bi-syndetic}. Hence, the modular map $\Delta_G$ is trivial.
 \end{proof}
 
 Since $\star$-approximate lattices generalise lattices in locally compact groups, one would like to think of them as discrete approximate subgroups with finite co-volume. We achieve this by showing that $\star$-approximate lattices admit a  finite volume fundamental set of sorts:
 
 \begin{proposition}\label{Proposition: star-Approximate lattices have finite co-volume}
  Let $\L$ be a $\star$-approximate lattice in a locally compact second countable group $G$. Then there is a Borel subset $B \subset G$ with finite Haar measure such that $B\L^2=G$ and $B^{-1}B \cap \L^2 = \{e\}$. 
 \end{proposition}

 \begin{proof}
    Let $\nu$ be a proper $G$-invariant Borel probability measure on $\Omega_{\L}$ and let $\mu_G$ denote the Haar-measure $(\mathcal{P}_{\L})^*\nu$. Let $V$ be a compact neighbourhood of the identity such that $V^{-1}V \cap \L^2 =\{e\}$ and let $(g_n)_{n \geq 0}$ be a sequence of elements of $G$ such that $G=\bigcup_{n \geq 0} g_nV$. Define inductively $B_n:=g_nV \setminus \bigcup_{m<n} B_m\L^2$ and $B:=\bigcup_{n\geq 0}B_n$. Then $B^{-1}B \cap \L^2 =\{e\}$ and $B\L^2=G$. Now the Haar measure $\mu_G$ is inner regular so there is a sequence of compact subsets $K_n \subset B$ with $\sup_{n \geq 0} \mu_G(K_n) = \mu_G(B)$. But $\mu_G(K_n) \leq 1$ for all integers $n \geq 0$ by Lemma \ref{Lemma: Measure of pull-backs of small neighbourhoods through the periodization map} so $\mu_G(B) \leq 1$.
 \end{proof}
 
 \begin{remark}
 Proposition \ref{Proposition: star-Approximate lattices have finite co-volume} shows that $\star$-approximate lattices are, in particular, approximate lattices in the sense of Hrushovski \cite{Hrushovski2020quasimodels}.
 \end{remark}
 
 \subsubsection{Hulls and commensurability}
 We prove now a result, used extensively in the rest of the paper, that can be seen as a version of Ruzsa's covering lemma (see e.g. \cite{tointon_2019}) for $\star$-approximate lattices. Furthermore, the proof provides a good illustration of the strategy used later on to prove Proposition \ref{Proposition: Range of a cocycle associated to a section of an extended hull}.

 \begin{proposition}\label{Proposition: Towers of star-approximate subsets are commensurable, general form}
  Let $X_0$ be a uniformly discrete subset of a locally compact second countable group $G$ and let $Y$ be a subset such that $X_0Y$ is uniformly discrete. Suppose that there is a proper $G$-invariant Borel probability measure on $\Omega_{X_0}^{ext}$. Then $Y$ is covered by finitely many right-translates of $X_0^{-1}X_0$.  
 \end{proposition}

  \begin{proof}
  Let $F \subset Y$ be such that the subsets $(X_0f)_{f \in F}$ are disjoint. Since $X_0F \subset X_0Y$ there is a neighbourhood $V \subset G$ of the identity such that the subsets $(X_0fv)_{f \in F, v \in V}$ are disjoint. Hence, 
  $$X_0^{-1}X_0 \cap FVV^{-1}F^{-1} =\{e\}.$$ 
  By Lemma \ref{Lemma: Measure of pull-backs of small neighbourhoods through the periodization map} we thus have $(\mathcal{P}_{X_0})^*\nu(V^{-1}F^{-1})\leq 1$. Since $\nu$ is $G$-invariant the measure $(\mathcal{P}_{X_0})^*\nu$ is a left-Haar measure on $G$ by Corollary \ref{Corollary: Pull-back through periodization of G-invariant measure is a Haar-measure}. But $G$ is unimodular according to Proposition \ref{Proposition: Envelope of a star-approximate lattice is unimodular} so $(\mathcal{P}_{X_0})^*\nu$ is also right-invariant. Therefore, 
  $$1 \geq (\mathcal{P}_{X_0})^*\nu(V^{-1}F^{-1}) =(\mathcal{P}_{X_0})^*\nu(V^{-1})\vert F\vert$$
  since the subsets $(fV)_{f \in F}$ are disjoint. The subset $F$ is thus finite and according to Ruzsa's covering lemma $Y$ is covered by finitely many right-translates of $X_0^{-1}X_0$. 
 \end{proof}

 \begin{corollary}\label{Corollary: Towers of star-approximate subsets are commensurable}
  If $\L_1 \subset \L_2$ are two $\star$-approximate lattices in a locally compact second countable group $G$, then $\L_1$ and $\L_2$ are commensurable.
 \end{corollary}
 
 \begin{corollary}\label{Corollary: Approximate subgroup condition is shallow}
  Let $X_0$ be a  closed subset of a locally compact second countable group $G$ such that $(X_0^{-1}X_0)^3$ is uniformly discrete. If there is a proper $G$-invariant Borel probability measure on $\Omega_{X_0}^{ext}$ (or $\Omega_{X_0}$), then $X_0^{-1}X_0$ is an approximate subgroup, hence a $\star$-approximate lattice.
 \end{corollary}
 
 \begin{proof}
 In particular, $X_0(X_0^{-1}X_0)^2$ is uniformly discrete. By Proposition \ref{Proposition: Towers of star-approximate subsets are commensurable, general form} applied to $X_0$ and $(X_0^{-1}X_0)^2$, $(X_0^{-1}X_0)^2$ is contained in finitely many right translates of $X_0^{-1}X_0$. Since $X_0^{-1}X_0$ is symmetric, $X_0^{-1}X_0$ is an approximate subgroup. But $\Omega_{X_0}^{ext} \subset \Omega_{X_0^{-1}X_0}^{ext}$ and $\Omega_{X_0}^{ext}$ admits a proper $G$-invariant Borel probability measure. So $X_0^{-1}X_0$ is a $\star$-approximate lattice. 
 \end{proof}
 
 \begin{remark}
 The proof of Corollary \ref{Corollary: Approximate subgroup condition is shallow} uses only that $X_0(X_0^{-1}X_0)^2$ is uniformly discrete, rather than $(X_0^{-1}X_0)^3$. It is therefore natural to wonder whether the exponent $3$ is minimal or not. A beautiful result of Lagarias \cite{MR1400744} suggests that $3$ could be replaced by $1$ with some effort. We investigate this in a forthcoming paper. 
 \end{remark}
 
 Corollary \ref{Corollary: Approximate subgroup condition is shallow} is striking as it shows that the approximate subgroup assumption is a natural one when studying the dynamics in the Chabauty-Fell topology of discrete subsets of a group.

  \subsubsection{Upgrading model sets} In the proof of Theorem \ref{Main theorem} we will embed a $\star$-approximate lattice $\L$ in a certain specific way in an arithmetic group $\Gamma$ - thus generating some discrete Zariski-dense subgroup. It will be key to be able to show that the subgroup generated by $\L$ has, in fact, finite index in $\Gamma$:  
 
 \begin{lemma}\label{Lemma: A weak model set with finite co-volume is a model set}
  Let $\Gamma$ be a discrete subgroup of a product of locally compact second countable groups $G \times H$. Suppose that $\Gamma$ projects densely to $H$ and that $G$ is unimodular. Take a compact neighbourhood of the identity $W_0$ in $H$ and suppose that there is $F \subset G$ with finite Haar-measure such that $F \cdot p_G\left(\Gamma \cap \left( G \times W_0 \right)\right) = G$ where $p_G: G \times H \rightarrow G$ is the natural projection. Then $\Gamma$ is a lattice in $G \times H$. 
 \end{lemma}
 
 \begin{proof}
  We know that for all $(g,h) \in G \times H$ there is $\gamma_1 \in \Gamma$ such that $(g,h)\gamma_1^{-1} \in G \times W_0$. By assumption we now have $\gamma_2 \in \Gamma \cap (G \times W_0)$ such that $p_G((g,h)\gamma_1^{-1}\gamma_2^{-1}) \in F$. Therefore, we have $(g,h) \in \left(F \times W_0W_0^{-1}\right)\Gamma$. But $F \times W_0W_0^{-1}$ has finite Haar measure, so $\Gamma$ is a lattice. 
 \end{proof}

 \begin{corollary}\label{Corollary: Model susbets with finite co-volume are model sets}
   Let $\Gamma$ be a discrete subgroup of a product  of locally compact second countable groups $G \times H$. Suppose that $\Gamma$ projects densely to $H$ and injectively to $G$.  Take a compact neighbourhood of the identity $W_0$ in $H$ and suppose that $p_G\left(\Gamma \cap \left( G \times W_0 \right)\right)$ contains a $\star$-approximate lattice. Then $\Gamma$ is a lattice in $G \times H$. 
 \end{corollary}
 
 \begin{proof}
   Proposition \ref{Proposition: Envelope of a star-approximate lattice is unimodular} shows that $G$ is unimodular and Proposition \ref{Proposition: star-Approximate lattices have finite co-volume} proves that there exists $F \subset G$ as in Lemma \ref{Lemma: A weak model set with finite co-volume is a model set}. The subgroup $\Gamma$ is therefore a lattice by Lemma \ref{Lemma: A weak model set with finite co-volume is a model set}. 
 \end{proof}
 
 \begin{remark}
  Lemma \ref{Lemma: A weak model set with finite co-volume is a model set} is also discussed in more details in \cite{machado2019goodmodels} along with several questions around this topic. 
 \end{remark}

 \section{Cocycles and $\star$-Approximate Lattices}\label{Section: Cocycles and star-Approximate Lattices}
 \subsection{Cocycles Associated to Sections of Extended Hulls} We will now define certain cocycles that will mimic properties of cocycles on transitive spaces as studied by Mackey \cite{MR44536} and Zimmer \cite{zimmer2013ergodic}. A classical route to build cocycles consists in using \emph{sections} (see for instance \cite{MR44536,zimmer2013ergodic}, or \cite{bjorklund2017analytic} for a framework closer to ours). We start by defining a suitable notion of sections of the set $\mathcal{C}(G)$. 
 
 \begin{definition}\label{Definition: Section of set of closed subsets}
 Let $G$ be a locally compact group and $\mathcal{B}$ be a Borel subset of $\mathcal{C}(G)$. A \emph{Borel section} of $\mathcal{B}$ is a Borel map $s: \mathcal{B} \rightarrow G$ such that for any $ X \in \mathcal{B}$ we have $s(X) \in X$.  
 \end{definition}
 
 Precisely, the cocycles we are interested in are:
 
 \begin{lemma}\label{Lemma: Cocycle associated to a Borel section}
 Let $X_0$ be a uniformly discrete subset of a locally compact group $G$. Let $s: \Omega_{X_0}^{ext}\setminus\{\emptyset\} \rightarrow G$ be a Borel section. Then the map 
 \begin{align*}
  \alpha_s : G \times (\Omega_{X_0}^{ext}\setminus\{\emptyset\}) & \longrightarrow G \\
             (g, X) & \longmapsto s(gX)^{-1}gs(X)
 \end{align*}
 is a strict Borel cocycle that takes values in $\overline{X_0^{-1}X_0}$. 
 \end{lemma}

 \begin{proof}
  The map $\alpha_s$ is Borel since $s$ is Borel. For all $g,h\in G$ and $X \in \Omega_{X_0}^{ext}\setminus\{\emptyset\}$ we have 
  $$ \alpha_s(g,hX)\alpha_s(h,X) = s(ghX)^{-1}gs(hX)s(hX)^{-1}hs(X) = s(ghX)^{-1}ghs(X)= \alpha_s(gh,X).$$
  So $\alpha_s$ is a strict cocycle. Moreover, we have 
  $$ \alpha_s(g,X) = s(gX)^{-1}gs(X) \in (gX)^{-1}gX \subset X^{-1}X \subset \overline{X_0^{-1}X_0}.$$
 \end{proof}
 
  It remains only to prove the existence \emph{Borel sections} of extended invariant hulls.
 
 \begin{proposition}\label{Proposition: Extended hulls of uniformly discrete subsets have Borel sections}
  Let $X_0$ be a uniformly discrete subset of a locally compact second countable group $G$. Then there is a Borel section $s: \Omega_{X_0}^{ext}\setminus\{\emptyset\} \rightarrow G$. 
 \end{proposition}
 
 \begin{proof}
  Since $X_0$ is uniformly discrete there is an open neighbourhood $V$ of the identity in $G$ such that $X_0^{-1}X_0 \cap V^{-1}V = \{e\}$. Thus, we have $|X \cap gV| \leq 1$ for any $X \in \Omega_{X_0}^{ext}$ and any $g \in G$. If $X \in U^{gV} \cap \Omega_{X_0}^{ext}$, then define $s_g(X)$ as the unique element of the subset $X \cap gV$. The maps $s_g: U^{gV} \cap \Omega_{X_0}^{ext} \rightarrow G$ are well-defined. Moreover, if $W \subset G$ is any open subset, then $s_g^{-1}(W)=U^{gV\cap W} \cap \Omega_{X_0}^{ext}$. So the maps $s_g$ are continuous. Take a sequence $(g_i)_{i \geq 0}$ such that $\bigcup_{i \geq 0} g_iV = G$. Then $\bigcup_{i\geq 0} U^{g_i V} = \mathcal{C}(G) \setminus\{\emptyset\}$. Now define $B_i:= \Omega_{X_0}^{ext} \cap \left(U^{g_iV} \setminus \left( \bigcup_{0 \leq j < i}U^{g_jV}\right)\right)$ for all integers $i \geq 0$ and $s: \Omega_{X_0}^{ext} \setminus\{\emptyset\} \rightarrow G$ as the unique map such that $s_{|B_i}=\left(s_{g_i}\right)_{|B_i}$. Since the subsets $B_i$ are Borel and the maps $s_g$ are continuous, the map $s$ is Borel. Moreover, for all $X \in \Omega_{X_0}^{ext} \setminus\{\emptyset\}$ we indeed have $s(X) \in X$. 
 \end{proof}

 \subsection{A Reduction Lemma}
 Before we move on with our study of the cocycles defined in the previous section, we make an elementary observation about the range of Borel maps that satisfy some functional equation with respect to a cocycle. To this end we first recall a well known fact about ergodicity and dense subgroups:
 
 \begin{lemma}\label{Lemma: Action of a dense subgroup is ergodic}
  Let $G$ be a locally compact group, $D \subset G$ be a dense subgroup and $X$ be a compact $G$-space. If $\nu$ is an ergodic $G$-invariant Borel probability measure, then the $D$-action on $(X,\nu)$ is also ergodic. 
 \end{lemma}

 We now prove the fact that will be used as the starting point of the proof of Proposition \ref{Proposition: Range of a cocycle associated to a section of an extended hull} in Subsection \ref{Subsection: Range of Cocycles and Compact Cocycles}.
 
 \begin{lemma}\label{Lemma: Cohomology reduction lemma}
 Let $G$ and $H$ be two locally compact second countable groups. Let $X$ be a compact metric $G$-space and $\nu$ be a $G$-invariant ergodic Borel probability measure on $X$. Let $\alpha: G \times X \rightarrow H$ be a Borel cocycle that take values in a subset $A \subset H$ and let $B \subset H$ be another subset. Suppose that there is a Borel map $\phi : X \rightarrow H$ such that for all $g \in G$ and $\nu$-almost every $x \in X$ we have 
 $$ \phi(gx) = \alpha(g,x)\phi(x)B. $$
 Then there is $h \in H$ such that for every neighbourhood of the identity $V$ in $H$ we have 
 $$ \nu-a.e.\  x \in X, \phi(x) \in AVhB^{-1}.$$
 \end{lemma}
 
 \begin{proof}
  Consider the measure $\phi_*\nu$. It is a Borel probability measure on the locally compact second countable group $H$. So $\phi_*\nu$ has a well-defined non-trivial support $S \subset H$ and we can choose $h \in S$. Then for any open neighbourhood $V$ of $e$ we have $$\phi_*\nu(Vh) = \nu(\phi^{-1}(Vh))> 0.$$ So for all $g \in G$ and $\nu$-almost every $x \in \phi^{-1}(Vh)$ we have 
  $$\phi(gx) \in AVhB^{-1}.$$
  But for any countable dense subgroup $D$ of $G$ there is $Y \subset  \phi^{-1}(Vh)$ with $\nu(Y)=\nu(\phi^{-1}(Vh))> 0$ such that for all $x \in Y$ and all $d \in D$ we have 
  $\phi(dx) \in AVhB^{-1}$. So $DY \subset \phi^{-1}\left(AVhB^{-1}\right)$. And we know that $\nu(DY) = 1$ by Lemma \ref{Lemma: Action of a dense subgroup is ergodic}. 
 \end{proof}

 Likewise, we prove a result for cocycles taking values in unitary groups that will be at the heart of our investigation of property (T) for $\star$-approximate lattices.
 
 \begin{lemma}\label{Lemma: Cohomology reduction lemma 2}
  Let $G$ be a locally compact group that acts continuously on a compact metric space and let $\nu$ be a $G$-invariant ergodic Borel probability measure on $X$. Let $\mathcal{H}$ be a separable Hilbert space and let $\alpha: X \times G \rightarrow U(\mathcal{H})$ be a Borel cocycle with target a subset $A$ of the unitary group of $\mathcal{H}$. Suppose that there exists a Borel map $\phi: X \rightarrow \mathcal{H}$ such that for all $g \in G$ and $\nu$-almost every $x \in X$ we have 
 $$ \phi(gx) = \alpha(g,x)\phi(x). $$
 Then there is $\xi \in \mathcal{H}$ such that for every $\epsilon > 0$ we have
 $$ \nu-a.e.\  x \in X, \phi(x) \in A(B(\xi,\epsilon))$$
 where $B(\xi, \epsilon)$ denotes the ball of centre $\xi$ and radius $\epsilon$. 
 \end{lemma}

 \begin{proof}
  Consider the measure $\phi_*\nu$. It is a Borel probability measure on the separable Hilbert space $\mathcal{H}$ whose (strong) topology is second countable. So $\phi_*\nu$ has a well-defined non-trivial support $S \subset \mathcal{H}$ and we can choose $\xi \in S$. Then for any $\epsilon > 0$ we have $\nu(\phi^{-1}(B(\xi,\epsilon))) > 0$. For all $g \in G$ and $\nu$-almost every $x \in \phi^{-1}(B(\xi,\epsilon))$ we have 
  $$\phi(gx) = \alpha(g,x)\phi(x) \in A(B(\xi,\epsilon)).$$ And we conclude as above. 
 \end{proof}

 \subsection{Range of Cocycles and Compact Cocycles}\label{Subsection: Range of Cocycles and Compact Cocycles}
 Throughout this section we fix a $\star$-approximate lattice $\L$ in a locally compact group $G$ and a proper $G$-invariant ergodic Borel probability measure $\nu_0$ on $\Omega_{\L}^{ext}$. Let $s : \Omega_{\L}^{ext}\setminus\{\emptyset\} \rightarrow G$ be a Borel section given by Proposition \ref{Proposition: Extended hulls of uniformly discrete subsets have Borel sections} and let $\alpha_s: G \times \Omega_{\L}^{ext}\setminus\{\emptyset\} \rightarrow G $ be the strict Borel cocycle associated to $s$ (Lemma \ref{Lemma: Cocycle associated to a Borel section}). Since $\nu_0(\{\emptyset\})=0$ the cocycle $\alpha_s$ gives rise to a Borel cocycle, that we denote by $\alpha_s$ as well, over $(\Omega_{\L}^{ext},\nu_0)$. Finally, we fix an abstract group homomorphism $T: \langle \Lambda \rangle \rightarrow H$ with target a locally compact second countable group. Our goal is to prove the following result linking the set $T(\L)$ to the range of the cocycle $T\circ \alpha_s$. 
 
 \begin{proposition}\label{Proposition: Range of a cocycle associated to a section of an extended hull}
  Suppose that the Borel cocycle $T\circ \alpha_s$ is cohomologous to a cocycle $\beta$ that takes values in a subset $L \subset H$. Then there is $h_0 \in H$ such that for every neighbourhood  of the identity $V$ in $H$ there is a compact subset $K \subset H$ with
  $$ T(\L) \subset Vh_0(L^{-1}L)^2h_0^{-1}K.$$
 \end{proposition}
 
 The proof reduces to two steps. First, we build a cocycle $\alpha'$ cohomologous to $\alpha_s$, taking values in some power of $\L$ (here $\L^6$) but such that $T \circ \alpha'$ takes values in some thickening of $L^{-1}L$. 
 
 \begin{lemma}\label{Lemma: First reduction of the cocycle T circ alpha}
  There is $h_0 \in H$ such that for every neighbourhood of the identity $V \subset H$ we can find Borel maps $f: \Omega_{\L}^{ext} \rightarrow G$ and $h :  \Omega_{\L}^{ext} \rightarrow H$ defined $\nu_0$-almost everywhere such that for all $g \in G$ and $\nu_0$-almost every $X \in \Omega_{\L}^{ext}$ we have 
    $$\alpha_f(g,X):=f(gX)^{-1}gf(X) \in \L^6 \text{ and } T(\alpha_f(g,X)) \in Vh_0L^{-1}Lh(X)^{-1}.$$
  \end{lemma}

 \begin{proof}
 First of all, note that $\alpha_s$ takes values in $\L^2$ by Lemma \ref{Lemma: Cocycle associated to a Borel section} and that the subset $\L^2$ is countable. So $T \circ \alpha_s$ is a well-defined Borel cocycle that takes values in $T(\L^2)$. Choose a measurable map $\phi: \Omega_{\L}^{ext} \rightarrow H$ such that for all $g \in G$ and $\nu_0$-almost every $X \in \Omega_{\L}^{ext}$ we have
  $$ T\circ\alpha_s(g,X)= \phi(gX)\beta(g,X)\phi(X)^{-1}.$$
  By Lemma \ref{Lemma: Cohomology reduction lemma} there is $h_0 \in H$ such that for every neighbourhood $V$ of the identity in $H$ and $\nu_0$-almost every $X \in \Omega_{\L}^{ext}$ we have 
  $$ \phi(X) \in T(\L^{2})Vh_0L^{-1}.$$
  
 Consider an enumeration $(\lambda_n)_{n\geq 0}$ of $\L^2$ and define 
  \begin{align*}
   p_V: T(\L^2)Vh_0L^{-1} & \longrightarrow \L^2 \\
              g      & \longmapsto \lambda_{\min\left\{ n \in \N \vert g \in T(\lambda_n)Vh_0L^{-1} \right\}}.
  \end{align*}
  Since $p_V^{-1}\left(\{\lambda_0,\ldots, \lambda_n\}\right)= \bigcup_{0\leq i \leq n} T(\lambda_n)Vh_0L^{-1}$ the map $p_V$ is Borel measurable. We can thus define the map $h : \Omega_{\L}^{ext} \rightarrow Vh_0L^{-1}$ for $\nu_0$-almost every $X \in \Omega_{\L}^{ext}$ by $\phi(X)=T(p_V(\phi(X)))h(X)$. We have for all $g \in G$ and $\nu_0$-almost every $X \in \Omega_{\L}^{ext}$,
  $$T\circ\alpha_s(g,X) = T\left(p_V(\phi(gX))\right)h(gX)\beta(g,X)h(X)^{-1}T(p_V(\phi(X))^{-1}).$$
  So 
  $$T(f(gX)^{-1}gf(X)) = h(gX)\beta(g,X)h(X)^{-1},$$
  where $f$ is defined for $\nu_0$-almost every $X \in \Omega_{\L}^{ext}$ by $f(X):=s(X)p_V(\phi(X))^{-1}$. Then one checks that $f$ and $h$ work.
  \end{proof}
 
 We then show that the range of such a cocycle $\alpha_f$ must be large in some power of $\L$, thus proving that the set of elements $\lambda \in \L$ such that $T(\lambda)$ belongs to some thickening of $L^{-1}L$ is large. However, the range of $\alpha_f$ as defined by the map $\Omega_{\L,G}^{ext} \rightarrow \mathcal{C}(G)$ given by $X \mapsto \{\alpha_f(g,X) | g \in G \}$ might not be well-behaved (e.g. non-Borel) so we proceed more carefully. 

 \begin{proposition}\label{Proposition: Ruzsa's covering Lemma for cocycles}
 Let $f: \Omega_{\L,G}^{ext} \rightarrow G$ be a Borel measurable map and let $\alpha_f: G \times \Omega_{\L,G}^{ext} \rightarrow G$ denote the cocycle defined by $\alpha_f (g,X)= f(gX)^{-1}gf(X)$ for all $g \in G$ and $\nu_0$-almost all $X \in  \Omega_{\L,G}^{ext}$. Suppose that $\Delta: \Omega_{\L,G}^{ext} \rightarrow \mathcal{C}(G)$ is a Borel measurable map defined $\nu_0$-almost everywhere such that :
 \begin{enumerate}
  \item $\alpha_f$ takes values in $\L^n$ for some integer $n \geq 0$;
  \item for all $g \in G$ and $\nu_0$-almost all $X \in \Omega_{\L,G}^{ext}$ we have $\alpha_f(g,X) \in \Delta(X)$.
 \end{enumerate}
Then there is a finite subset $F \subset \L$ such that $$ \L \subset \bigcup_{X \in \Omega_{\L,G}^{ext}} \Delta(X)\Delta(X)^{-1}F$$ where we set $\Delta(X)\Delta(X)^{-1}=\emptyset$ when $\Delta(X)$ is not defined. 
\end{proposition}
 We will also make an independent use of this proposition in our study of property (T). 
  \begin{proof} We first show that we can suppose that $\Delta$ takes values in $\Omega_{\L^n, G}^{ext}$. The map $\Delta':  \Omega_{\L,G}^{ext} \rightarrow \mathcal{C}(G)$ defined by $X \rightarrow \Delta(X) \cap \L^n$ satisfies $(\Delta')^{-1}(U_K) = \Delta^{-1}(\L^n \cap K)$ for all compact subsets $K \subset G$. So $\Delta'$ is Borel by Lemma \ref{Lemma: Criterion Borel measurability 2}. Since in addition  $\alpha_f(g,X) \in \Delta(X) \cap \L^n$ for all $g \in G$ and $\nu_0$-almost all $X \in \Omega_{\L,G}^{ext}$ we get the desired result. So suppose from now on that $\Delta$ takes values in $\Omega_{\L^n, G}^{ext}$. Let $\Phi: \Omega_{\L}^{ext} \longrightarrow \Omega_{\L^n}^{ext}$ be the map defined by $X \longmapsto f(X)\Delta(X)^{-1}$. We know that $\Phi$ is Borel and well-defined $\nu_0$-almost everywhere since both $f$ and $\Delta$ are.  The crux of the proof is to deduce a simple equation involving $f$, $\Delta$ (or equivalently $\Phi$) and intertwining the $G$-actions on $\Omega_{\L,G}^{ext}$ and $\Omega_{\L^n,G}^{ext}$. We know that for all $g \in G$ and $\nu_0$-almost all $X \in \Omega_{\L,G}^{ext}$ we have $$\alpha_f(g,X) \in \Delta(X).$$
  So we obtain 
  \begin{equation}
   f(gX) \in gf(X)\Delta(X)^{-1} = g \Phi(X). \label{Equation: eq1} \tag{$*$}
  \end{equation}
  We will now use (\ref{Equation: eq1}) to show the existence of a real number $\epsilon > 0$ such that $(\mathcal{P}_{\L^n})^{*}\left(\Phi_{*}\nu_0\right) \geq \epsilon \mu_G$ where $\mu_G$ is a Haar-measure on $G$ (fixed from now on) and $\mathcal{P}_{\L^n}$ denotes the periodization map introduced in Definition \ref{Definition: Periodization map}. So let $\nu_1$ denote the push-forward of $\nu_0$ by $\Phi$. For all open subsets $W \subset G$ and all elements $g \in G$ we have 
 \begin{align*}
  \nu_1\left(gU^{W}\right)& = \int_{\Omega_{\L}^{ext}} \mathds{1}_{U^{W}}\left(g^{-1}\Phi(X)\right)d\nu_0(X) \\
                   & \geq \int_{\Omega_{\L}^{ext}} \mathds{1}_{W}(f(g^{-1}X))d\nu_0(X) \\
                   & = \int_{\Omega_{\L}^{ext}} \mathds{1}_{W}(f(X))d\nu_0(X)  = f_*\nu_0(W).
 \end{align*}
 where the second-to-last line is implied by (\ref{Equation: eq1}) and the last one is obtained by $G$-invariance of $\nu_0$. Since $gU^{W} = U^{gW}$, the above inequalities imply that for all $h \in G$ and all open subsets $W \subset G$ we have $\nu_1(U^{W}) \geq \mu_0(hW)$ where $\mu_0:=f_*\nu_0$. Define $\mu_1:=(\mathcal{P}_{\L^n})^*\nu_1$ where $\mathcal{P}_{\L^n}$ is the periodization map. Since $\L^n$ is uniformly discrete there is an open neighbourhood of the identity $W \subset G$ such that $W^{-1}W \cap \L^{2n}=\{e\}$. By Lemma \ref{Lemma: Measure of pull-backs of small neighbourhoods through the periodization map} for all $g, h\in G$ we have $$\mu_1(gW) = \nu_1(U^{gW}) \geq \mu_0(hW).$$ Furthermore, given any Borel probability measure $\mu$ on $G$ we have $$ \mu_1(gW)\geq \mu * \mu_0 (hW).$$
 In fact, the above proves that for any $W' \subset W$ open we have
 $$\mu_1(gW') \geq \mu * \mu_0 (hW') .$$
 By outer regularity of finite Borel measures on second countable spaces, for any $B \subset W$ Borel we thus have, 
 $$\mu_1(gB) \geq \mu * \mu_0 (hB).$$
 Now take $\mu$ that has density a continuous compactly supported function $\delta: G \rightarrow \R$ with respect to $\mu_G$ on $G$. We know that $\mu * \mu_0$ is absolutely continuous with respect to $\mu_G$ and has continuous density, $\rho$ say, with respect to $\mu_G$. Since $\rho$ is a non-trivial non-negative continuous function we can find $h \in G$, a neighbourhood of the identity $W' \subset W$ and a real number $\epsilon > 0$ such that $\rho_{\vert hW'} \geq \epsilon > 0$. Hence, for all $g \in G$ and all Borel subsets $B \subset gW'$  we have 
 $$ \mu_1(B) \geq \mu * \mu_0(hg^{-1}B) \geq \epsilon\mu_G(hg^{-1}B) = \epsilon\mu_G(B).$$
 But $G$ is second countable so we can find a sequence $(g_i)_{i \geq 0}$ of elements of $G$ and a countable Borel partition $(W_i)_{i\geq 0}$ of $G$ such that $W_i \subset g_iW'$ for all integers $i \geq 0$. Thus, for all Borel subsets $B \subset G$ we have 
 $$ \mu_1(B) =  \sum_{i\geq 0}\mu_1\left(B\cap W_i\right) \geq\epsilon \sum_{i\geq 0}\mu_G\left(B\cap W_i\right) = \epsilon \mu_G(B). $$
 We will now prove Proposition \ref{Proposition: Ruzsa's covering Lemma for cocycles} as a consequence of a version of Ruzsa's covering lemma. Let $\mathcal{B} \subset \Omega_{\L}^{ext}$ be a co-null Borel subset such that $\Phi$ is well-defined for every $X \in \mathcal{B}$. Let $F \subset \L$ be such that for every $X \in \mathcal{B}$ the subsets $(\Phi(X)f)_{f \in F}$ are pairwise disjoint. Choose moreover a symmetric neighbourhood of the identity $W \subset G$ such that $W^{-1}W \cap \L^{2n+2} = \{e\}$. We know that for $X \in \mathcal{B}$ we have $$F^{-1}\left(\Phi(X)\right)^{-1}\left(\Phi(X)\right) F \subset \L^{2n+2},$$ 
 so $$ W^{-1}W \cap F^{-1}\Phi(X)^{-1}\Phi(X) F = \{e\}.$$
 Therefore, 
  $$FW^{-1}WF^{-1} \cap \Phi(X)^{-1}\Phi(X) = FF^{-1} \cap \Phi(X)^{-1}\Phi(X).$$
 But the subsets $(\Phi(X)f)_{f \in F}$ are pairwise disjoint so 
 $$ FW^{-1}WF^{-1} \cap \Phi(X)^{-1}\Phi(X) = \{e\}.$$
 Recall that we defined $\mu_1 = (\mathcal{P}_{\L^n})^{*}\nu_1$ so that, by Lemma \ref{Lemma: Measure of pull-backs of small neighbourhoods through the periodization map} and the above discussion, we have 
 $$ 1 \geq \nu_1(U^{WF^{-1}}) = \mu_1(WF^{-1}) \geq \epsilon\mu_G(WF^{-1}).$$
  But  $G$ is unimodular and $F^{-1} \subset \L \subset \L^{n+1}$, so 
 $$1 \geq \epsilon\mu_G(WF^{-1}) = \sum_{f\in F} \epsilon \mu_G(Wf^{-1})=\epsilon\left| F \right| \mu_G(W).$$
 We thus have 
 $$ \left|F\right| \leq \frac{1}{\epsilon \mu_G(W)} < \infty. $$
 So take one such $F$ maximal for the inclusion. The subset $F$ is finite and for all $\lambda \in \L$ there are $f \in F$ and $X \in \mathcal{B}$ such that $\Phi(X)\lambda \cap \Phi(X)f \neq \emptyset$ i.e. $\lambda \in \Phi(X)^{-1}\Phi(X)f$. But $\Phi(X) = f(X)\Delta(X)^{-1}$ so we have the inclusion $\lambda \in \Delta(X)\Delta(X)^{-1} F$. 
 \end{proof}

\begin{proof}[Proof of Proposition \ref{Proposition: Range of a cocycle associated to a section of an extended hull}.]
  Let $h_0$, $f$ and $h$ be as in Lemma \ref{Lemma: First reduction of the cocycle T circ alpha} and define the map 
  \begin{align*}
   \Delta: \Omega_{\L}^{ext} & \longrightarrow \Omega_{\L^6}^{ext} \\
                  X             & \longmapsto \L^6\cap T^{-1}(Vh_0L^{-1}Lh(X))
  \end{align*}
  which is well-defined $\nu_0$-almost everywhere and such that for all $g \in G$ and $\nu_0$-almost every $X \in \Omega_{\L}^{ext}$ we have 
  $$ f(gX)^{-1}gf(X) \in \Delta(X). $$
  For any compact subset $K \subset G$ we have 
  \begin{align*}
   \Delta^{-1}\left(U_K\right) & = \{X \in \Omega_{\L}^{ext} \vert \Delta(X)\cap K = \emptyset \} \\
                               & = \{X \in \Omega_{\L}^{ext} \vert T(\L^6 \cap K) \cap Vh_0L^{-1}Lh(X) = \emptyset \} \\
                               & = \Omega_{\L}^{ext}\setminus h^{-1}\left(L^{-1}Lh_0^{-1}V^{-1}\left(T(\L^6 \cap K)\right)\right).
  \end{align*}
  So $\Delta$ is Borel measurable according to Lemma \ref{Lemma: Criterion Borel measurability 2}. We may now conclude the proof of Proposition \ref{Proposition: Range of a cocycle associated to a section of an extended hull}. Indeed, by Proposition \ref{Proposition: Ruzsa's covering Lemma for cocycles} we have $F \subset \L$ finite
\begin{align*}
 T(\L) &\subset T\left(\left(\bigcup_{X \in \Omega_{\L}^{ext}} \Delta(X)\Delta(X)^{-1}\right)F\right) \\
       &\subset \left(\bigcup_{X \in \Omega_{\L}^{ext}} T\left(\Delta(X)\Delta(X)^{-1}\right)\right)T(F) \\
       &\subset Vh_0(L^{-1}L)^2h_0^{-1}V^{-1}T(F).
\end{align*}
\end{proof}

 \begin{corollary}\label{Corollary: Compact cocycles}
  Suppose that the Borel cocycle $T\circ \alpha_s$ is cohomologous to a cocycle that takes values in a compact subgroup $L \subset H$. Then $T(\L)$ is a relatively compact subset of $H$.
 \end{corollary}

 \begin{proof}
Take a compact neighbourhood of the identity $V \subset H$. There are $h_0 \in H$ and $K \subset H$ given by Proposition \ref{Proposition: Range of a cocycle associated to a section of an extended hull} such that $T(\L) \subset Vh_0Lh_0^{-1}K$. But $Vh_0Lh_0^{-1}K$ is compact.
\end{proof}
 
 \subsection{Constant Cocycles}
 We end this section with an application of Proposition \ref{Proposition: Range of a cocycle associated to a section of an extended hull}. Along with Corollary \ref{Corollary: Compact cocycles} these will enable us to use the strength of Zimmer's cocycle superrigidity. 
 
 Let $\L,G,H,T$ and $s$ be as in Subsection \ref{Subsection: Range of Cocycles and Compact Cocycles}. Assume that the cocycle $T \circ \alpha_s$ is cohomologous to a constant cocycle associated to a continuous group homomorphism $\pi:G \rightarrow H$. We start by showing that in general a conjugate of $\pi$ extends $T$ modulo a two-sided error. 
 
 \begin{proposition}\label{Proposition: Two-sided error between T and extension}
  There is $h \in H$ such that for all $n \in \N$ and for all neighbourhoods of the identity $V \subset G \times H$ there exists a compact set $K_{V,n} \subset G \times H$ with
  $$ \forall \lambda \in \L^n,\ T(\lambda) \in V\pi^h(\lambda)K_{V,n} $$
  where $\pi^{h}(g):=h\pi(g)h^{-1}$. 
 \end{proposition}

 \begin{proof}
 Let us first prove the case $n=1$. Let $i: \langle \Lambda \rangle \rightarrow G$ denote the inclusion map and consider the diagonal map $i \times T : \langle \Lambda \rangle \rightarrow G \times H$. The cocycle $(i\times T)\circ \alpha_s : G \times \Omega_{\L}^{ext} \rightarrow G \times H$ is cohomologous to the constant cocycle associated to the continuous group homomorphism $\id \times \pi : G \rightarrow G\times H$. By Proposition \ref{Proposition: Range of a cocycle associated to a section of an extended hull} there is $(g,h) \in G\times H$ such that for all neighbourhoods of the identity $V \subset G \times H$ there is a compact subset $K \subset G \times H$ such that $$i \times T (\L) \subset V (g,h)\Gamma_{\pi}(g^{-1},h^{-1}) K= V \Gamma_{\pi^{h\pi(g)^{-1}}} K .$$
 Here, $\Gamma_{\pi} \subset G \times H$ denotes the graph of $\pi$ (equivalently, the image of $\id \times \pi$). 
 So a quick computation shows that for all neighbourhoods of the identity $W \subset H$ there is a compact subset $K_W \subset H$ such that for all $\lambda \in \L$ we have 
 $$ T(\lambda) \in W \pi^{h\pi(g)^{-1}}(\lambda) K_W.$$  
 The case $n > 1$ follows readily noticing that there is a finite subset $F \subset \langle \Lambda \rangle$ such that $\L^n\subset \L F$. 
 \end{proof}

 Following \cite[Def. V.3.5]{MR1090825} we will say that a continuous group homomorphism $\pi : G \rightarrow H$ \emph{almost extends} $T$ if $T(\gamma)\pi(\gamma)^{-1}$ is centralised by $\pi(G)$ for all $\gamma \in \langle \Lambda \rangle$. When $H$ is abelian the error term from Proposition \ref{Proposition: Two-sided error between T and extension} becomes simpler. So we find that $\pi^h$ almost extends $T$ and the error term is bounded on $\L$ (in the spirit of the extension results from \cite{moody1997meyer} and \cite{machado2020approximate}). To deal with the non-commutative case we first draw another formula from Proposition \ref{Proposition: Two-sided error between T and extension}:

 \begin{corollary}\label{Corollary: Corollary to Two-sided error between T and extension} 
  With $h$ given by Proposition \ref{Proposition: Two-sided error between T and extension}. Choose $\gamma \in \langle \Lambda \rangle$ and let $\delta(\gamma):= T(\gamma)^{-1}\pi^h(\gamma)$. Then for all neighbourhoods of the identity $V \subset H$ there is a compact subset $K$ such that 
  $$\forall g\in G, \delta(\gamma)\pi^{h}(g) \in V \pi^h(g) K.$$
 \end{corollary}

 \begin{proof}
  Take $\lambda \in \L$. By Proposition \ref{Proposition: Two-sided error between T and extension} we know that for all neighbourhoods of the identity $W \subset H$ there is a compact subset $K_W \subset H$ such that 
  $$ \pi^h(\gamma\lambda) \in W^{-1}T(\gamma \lambda)K_W^{-1} \subset W^{-1}T(\gamma)W\pi^{h}(\lambda)  K_WK_W^{-1}.$$
  So for all neighbourhoods of the identity $W' \subset H$ there is a compact subset $K' \subset H$ such that for all $\lambda \in \L$ we have
  $$ \delta(\gamma) \pi^h(\lambda) \in W'\pi^h(\lambda)K'.$$
  One can now deduce Corollary \ref{Corollary: Corollary to Two-sided error between T and extension} from Corollary \ref{Corollary: Star-approximate lattices are bi-syndetic}. 
 \end{proof}

  When $H$ acts on a CAT(0) space (see e.g. \cite{MR1744486}) this gives strong information on the action of the error term $\delta(\gamma)$ on the visual boundary:
  
 \begin{proposition}\label{Proposition: Consequence two-sided error in CAT(0) space}
  Let $X$ be a complete CAT(0) space and suppose that $H$ acts continuously by isometries on $X$. Take $x_0 \in X$ and let $Y$ be the closure of $\pi^h(G)\cdot x_0$ in $X \cup \partial X$ equipped with the cone topology. Then $\delta(\gamma)$ stabilises $Y \cap \partial X$ point-wise.
 \end{proposition}
 
 \begin{proof}
  Take $\xi \in Y \cap \partial X$ and let $(g_n)_{n\geq 0}$ be a sequence of elements of $G$ such that $\pi^h(g_n)\cdot x_0 \rightarrow \xi$ as $n$ goes to $\infty$. Choose a compact neighbourhood of the identity $V \subset H$ and a compact subset $K \subset H$ given by Corollary \ref{Corollary: Corollary to Two-sided error between T and extension}. For all $n \geq 0$ we can find $v_n \in V$ and $k_n \in H$ such that $\delta(\gamma)\pi^h(g_n) = v_n \pi^h(g_n) k_n$. Note that upon considering sub-sequences we can assume that $v_n \rightarrow v \in \overline{V}$. We now have  $$d_X(\delta(\gamma)\pi^h(g_n) \cdot x_0, v_n \pi^h(g_n) \cdot x_0)= d(x_0, k_n \cdot x_0)$$
  where $d_X$ denotes the distance on $X$. But $\pi^h(g_n)\cdot x_0 \rightarrow \xi \in \partial X$ and $k_n \in K$ compact so $\delta(\gamma)\cdot \xi = v \cdot \xi$. As this holds true for any $V$ we find $\delta(\gamma) \cdot \xi = \xi$. 
 \end{proof}

 Applying the above Proposition \ref{Proposition: Consequence two-sided error in CAT(0) space} to symmetric spaces and Bruhat--Tits buildings we are able to prove the main result of this section:
 
 \begin{proposition}\label{Proposition: Constant cocycles in semi-simple groups}
   Let $l$ be a local field and assume that $H$ is the group of $l$-points of an almost simple algebraic group $\mathbb{H}$ defined over $l$, and that $\pi$ has unbounded Zariski-dense image. Then $\pi^h$ almost extends $T$.
 \end{proposition}
 
\begin{proof}
 Fix $\gamma \in \langle \Lambda \rangle$. Let $X$ be: the symmetric space associated to $\mathbb{H}$ if $l$ is Archimedean, or the Bruhat--Tits building associated to $\mathbb{H}$ otherwise. Since $\pi(G)$ is unbounded and $X \cup \partial X$ is compact, we know by Proposition \ref{Proposition: Consequence two-sided error in CAT(0) space} that $\delta(\gamma)$ fixes point-wise a non-empty closed subset $Y \subset \partial X$ stable under the action of $\pi^h(G)$. But the point-wise stabiliser of $Y$, being the intersection of parabolic subgroups, is a Zariski-closed proper subgroup of $H$ and it is normalised by $\pi^h(G)$. It is thus central, and so $\delta(\gamma)$ is central as well. Since this is true for all $\gamma \in \langle \Lambda \rangle$ we find that $\pi^h$ almost extends $T$.
\end{proof}

 \section{Property (T) for Approximate Subgroups}\label{Section: Property T for Approximate Subgroups}
 
 We will now give a tentative definition of Property (T) for approximate subgroups. Our main goal in doing so is to prove that $\star$-approximate lattices in Kazhdan groups generate a finitely generated subgroup.
 
 \subsection{Definition and First Properties}\label{Subsection: Definition and First Properties}
 Recall that given a locally compact group $G$, a unitary representation $(\pi, \mathcal{H})$ and a subset $Q \subset G$, a $(Q,\epsilon)$-\emph{invariant vector} for some $\epsilon > 0$ is a unit vector $\xi \in \mathcal{H}$ such that $||\pi(g)\xi - \xi || < \epsilon$ for all $g \in Q$. If there are $(Q,\epsilon)$-invariant vector for all compact subsets $Q$ and all $\epsilon > 0$ we say that $\pi$ \emph{almost has invariant vectors}.
 
 \begin{definition}[Property (T) for approximate subgroups]\label{Definition: Property T for approximate subgroups}
  Let $\L$ be an approximate subgroup of some discrete group and $\langle \Lambda \rangle$ the group it generates. We say that $\L$ has \emph{property (T)} if there are a finite subset $Q \subset \L$ and $\epsilon > 0$ such that for any unitary representation $(\pi,\mathcal{H}_{\pi})$ of $\langle \Lambda \rangle$ that has $(Q,\epsilon)$-invariant vectors there is a unit vector $\xi \in \mathcal{H}_{\pi}$ such that $\pi(\L)(\xi)$ is totally bounded.
 \end{definition}
 
  Equivalence between Definition \ref{Definition: Property T for approximate subgroups} and the shorter definition given in the introduction is a clear consequence of the following lemma:
 
 \begin{lemma}\label{Lemma: Compactness criterion for the unitary image of an approximate subgroup} 
  Let $\L$ be an approximate subgroup in some group, let $(\pi, \mathcal{H}_{\pi})$ be a unitary representation of the group $\langle \Lambda \rangle$ it generates and let $\xi \subset \mathcal{H}_{\pi}$ be any vector. The following are equivalent:
  \begin{enumerate}
   \item for all $\delta > 0$ there is an approximate subgroup $\L(\delta,\xi)$ contained in and commensurable to $\L^2$ such that for all $\lambda \in \L(\delta,\xi)$ we have $||\pi(\lambda)(\xi) - \xi|| < \delta$;
   \item the subset $\pi(\L)(\xi)$ is totally bounded;
   \item there is a sub-representation $(\sigma, \mathcal{H}_{\sigma})$ with $\xi \in \mathcal{H}_{\sigma}$, and such that $\sigma(\L)$ is totally bounded in the strong topology.
  \end{enumerate}
  \end{lemma}

 \begin{proof}
  Let us start with (1) $\Rightarrow$ (2). Take $\delta >0$, the subset $\pi(\L(\delta,\xi)) (\xi)$ is contained in $B_{\mathcal{H}_{\pi}}(\xi, \delta)$ the ball of centre $\xi$ and radius $\delta$. But there is a finite subset $F_{\delta} \subset \L^3$ such that $\L \subset F_{\delta} \L(\delta,\xi)$ so $\pi(\L)(\xi) \subset \pi(F_{\delta})(B_{\mathcal{H}_{\pi}}(\xi, \delta)) \subset \bigcup_{f\in F_{\delta}} B_{\mathcal{H}_{\pi}}(\pi(f)(\xi),\delta)$. Therefore, $\pi(\L)(\xi)$ is covered by finitely many balls of radius $\delta$. The set $\pi(\L)(\xi)$ is thus totally bounded.
  
  Notice now that for any $\gamma \in \langle \Lambda \rangle$ there is $F_{\gamma} \subset \langle \Lambda \rangle$ finite such that $\L\gamma  \subset F_{\gamma}\L$. So $\pi(\L)(\pi(\gamma)(\xi))$ is totally bounded as well. Let $(\sigma,\mathcal{H}_{\sigma})$ denote the sub-representation of $\pi$ generated by $\xi$ i.e.  $\mathcal{H}_{\sigma}$ is the closure of the linear span of $\pi(\langle \Lambda \rangle)(\xi)$. We know that $\pi(\L)(\xi)$ contains a countable dense subset (it is totally bounded in a metric space), so we readily check that $\mathcal{H}_{\sigma}$ contains a countable dense subset $(\xi_n)_{n\geq 0}$ such that $\sigma(\L)(\xi_n)$ is relatively compact for all $n\geq 0$. By the Arzela-Ascoli theorem, since $\sigma(\L)$ is obviously a set of equicontinuous operators, we know that $\sigma(\L)$ is relatively compact in the point-wise topology. So (2) $\Rightarrow$ (3).
  
    Finally (3) $\Rightarrow$ (1): choose $\delta > 0$ and let $V_{\delta}(\xi)$ be the open subset of $U(\mathcal{H}_{\sigma})$ defined by $\{ T \in U(\mathcal{H}_{\sigma}) \left|\ ||T(\xi) - \xi || < \delta \right.\}$. Then $V_{\delta}(\xi)$ is symmetric and $V_{\delta}(\xi)^2 \subset V_{2\delta}(\xi)$. But $\sigma(\L)$ is covered by finitely many left-translates of $V_{\delta}(\xi)$. Hence, by the usual argument $\L^{2} \cap V_{2\delta}(\xi)$ is an approximate subgroup commensurable to $\L$ (see e.g. \cite[\S 2]{machado2019goodmodels}). So $\L(2\delta,\xi):= \L^{2} \cap V_{2\delta}(\xi)$ works.
 \end{proof}

 \begin{remark}
  As the proof shows, or since compact metric subsets are separable, we can always find $\sigma$ as in (2) such that $\mathcal{H}_{\sigma}$ is separable. 
 \end{remark}

 \subsection{Links with Relative Property (T)} For illustration purposes, we relate Definition \ref{Definition: Property T for approximate subgroups} to the relative property (T) of the pair $(\langle \Lambda \rangle,\L)$ (see \cite{MR2245534}). Note that we assume here that $\L$ is countable, but we will later see that this is always true (Proposition \ref{Proposition: Approximate subgroups with property T are finitely generated}).
 
 \begin{proposition}
  Let $\L$ be a countable approximate subgroup of some group. We have:
  \begin{enumerate}
   \item if $\L$ has property (T), then $(\langle \Lambda \rangle,\L)$ has the relative property (T);
   \item conversely, if $\L$ generates a lattice in an algebraic group over a local field and $(\langle \Lambda \rangle,\L)$ has the relative property (T), then $\L$ has property (T);
   \item if $G$ is a locally compact group, $N$ is a closed normal subgroup, $(G,N)$ has relative property (T), and $\Gamma \subset G$ is a lattice, then there is a compact subset $K \subset G$ such that for any compact neighbourhood of the identity $W$ containing $K$ the approximate subgroup $\Gamma \cap NW$ has property (T).
  \end{enumerate}
 \end{proposition}
 
 The proof follows from standard methods as found in \cite{MR2245534, MR2415834}. We do not include a proof since we do not use these facts in the following. Note moreover that it would be interesting to know if a general converse to (1) is true.

 \subsection{Heredity}
 We now turn to the heart of Section \ref{Section: Property T for Approximate Subgroups}: we will prove a certain heredity result about property (T) for $\star$-approximate lattices.
 
 \begin{proposition}\label{Proposition: Heredity of property (T)}
  Let $\L$ be a $\star$-approximate lattice in a locally compact second countable group $G$. The following are equivalent:
  \begin{enumerate}
   \item $G$ has property (T);
   \item $\L$ has property (T) for approximate subgroups.
  \end{enumerate}
 \end{proposition}

 \begin{proof}
  We begin with (2) $\Rightarrow$ (1) as the proof will use the ideas of Subsection \ref{Subsection: Definition and First Properties}. Let $(Q,\epsilon)$ be a Kazhdan pair for $\L$ and let $(\pi, \mathcal{H}_{\pi})$ be a unitary representation of $G$ that admits $(Q,\epsilon)$-invariant vectors. Let $\xi \in \mathcal{H}_{\pi}$ with $\pi(\Lambda)(\xi)$ totally bounded and $(\L(\xi, \delta))_{\delta \geq 0}$ be as in Lemma \ref{Lemma: Compactness criterion for the unitary image of an approximate subgroup} (1) and take $\delta > 0$. Take $V \subset G$  open such that $\pi(V) \subset V_{\delta/3}:= \{ T \in U(\mathcal{H}_{\sigma}) \left|\ ||T(\xi) - \xi || < \delta/3 \right.\}$. By Corollary \ref{Corollary: Star-approximate lattices are bi-syndetic} a finite number of left-translates of the subset $X(\delta,\xi) := V\L(\delta/3,\xi)V$ cover $G$. But then for all $g \in  X(\delta,\xi)$ we have $||\pi(g)\xi - \xi|| \leq \delta$. So there is a sub-representation  $(\sigma, \mathcal{H}_{\sigma})$ such that $\sigma(G)$ is relatively compact by Lemma \ref{Lemma: Compactness criterion for the unitary image of an approximate subgroup}. According to the Peter--Weyl theorem $\sigma$, and hence $\pi$, has a finite dimensional sub-representation. But that means that $G$ has property (T) according to the characterisation of property (T) from \cite[Th.2.12.4]{MR2415834}. 
  
  Conversely, let $\nu$ be a proper $G$-invariant ergodic measure on $\Omega_{\L,G}^{ext}$, let $(Q,\epsilon)$ be a Kazhdan pair for $G$ and let $(\pi, \mathcal{H}_{\pi})$ be a unitary representation of $\langle \Lambda \rangle$. Note that since $G$ is $\sigma$-compact (\cite[Th.1.3.1]{MR2415834}), $\langle \Lambda \rangle$ is countable. So we can assume that $\mathcal{H}_{\pi}$ is separable. Take a Borel section $s: \Omega_{\L,G}^{ext} \rightarrow G$ as in Definition \ref{Definition: Section of set of closed subsets} and let $\alpha_s: G \times \Omega_{\L,G}^{ext} \rightarrow \langle \Lambda \rangle$ be the Borel cocycle defined by $g,X \mapsto s(gX)^{-1}gs(X)$ (Lemma \ref{Lemma: Cocycle associated to a Borel section}). Thanks to the cocycle identity $\alpha_s(gh,X)=\alpha_s(g,hX)\alpha_s(h,X)$ we can define a unitary representation $\sigma$ of $G$ on $\mathcal{H}_{\sigma}:=L^2(\Omega_{\L,G}^{ext},\mathcal{H}_{\pi}; \nu)$ by $$\sigma(g)(f)( X) = \pi(\alpha_s(g^{-1},X)^{-1})(f(g^{-1}X))$$ (see e.g. \cite{zimmer2013ergodic} for this and more). Arguing exactly as in \cite[Th. 9.1.1]{zimmer2013ergodic} we show that there are a finite subset $Q' \subset \L$  and $\epsilon' >0$ - independent of $\pi$ - such that if $(\pi, \mathcal{H}_{\pi})$ has a $(Q',\epsilon')$-invariant vector, then $(\sigma,\mathcal{H}_{\sigma}) $ has a $(Q,\epsilon)$-invariant vector. So suppose from now on that $(\pi, \mathcal{H}_{\pi})$ has a $(Q',\epsilon')$-invariant vector. Since $(Q,\epsilon)$ is a Kazhdan pair for $G$ there is $\phi \in L^2(\Omega_{\L,G}^{ext},\mathcal{H}_{\pi}; \nu)$ with norm $1$ such that $\sigma(g)(\phi)=\phi$ for all $g \in G$. Therefore, for all $g \in G$ and $\nu$-almost all $X \in \Omega_{\L,G}^{ext}$ we have 
  $$\pi(\alpha_s(g^{-1},X)^{-1})(\phi(g^{-1}X)) = \phi(X).$$
  Note first that by ergodicity of the action of $G$ on $\Omega_{\L,G}^{ext}$ and since $\phi$ has norm $1$ we have that $\phi(X)$ has norm $1$ in $\mathcal{H}_{\pi}$ for $\nu$-almost all $X \in \Omega_{\L,G}^{ext}$. We will now proceed as in Section \ref{Section: Cocycles and star-Approximate Lattices} to produce a unit vector $\xi \in \mathcal{H}_{\pi}$ and a sequence of approximate subgroups $(\L(\xi, \delta))_{\delta > 0}$ commensurable to $\L$ such that $\pi(\L(\xi, \delta))(\xi) \subset B_{\mathcal{H}_{\pi}}(\xi, \delta)$  where $B_{\mathcal{H}_{\pi}}(\xi, \delta)$ is the ball centred at $\xi$ and of radius $\delta$ in $\mathcal{H}_{\pi}$. By Lemma \ref{Lemma: Cohomology reduction lemma 2} and since $\alpha_s$ takes values in $\L^2$ there is $\xi \in \mathcal{H}_{\pi}$ with norm $1$ such that for all $\delta > 0$ and $\nu$-almost all $X \in \Omega_{\L,G}^{ext}$ we have $\phi(X) \in \pi(\L^2)(B_{\mathcal{H}_{\pi}}(\xi, \delta))$. As in Lemma \ref{Lemma: First reduction of the cocycle T circ alpha} we will build $f_{\delta}: \Omega_{\L,G}^{ext} \rightarrow G$ and $\alpha_{f_{\delta}}$ defined for all $g \in G$ and almost all $X \in \Omega_{\L,G}^{ext}$ by $\alpha_{f_{\delta}}(g,X):=f_{\delta}(gX)^{-1}gf_{\delta}(X)$ such that $f_{\delta}$ is Borel, $\alpha_{f_{\delta}}$ takes values in $\L^6$ and $\pi\circ \alpha_{f_{\delta}}$ takes values in $V_{\delta}(\xi):=\{T \in U(\mathcal{H}_{\pi})\left| \ ||T( \xi) - \xi || < 2\delta\right.\}$. Choose an enumeration $(\lambda_n)_{n\geq 0}$ of $\Lambda^2$ and for every $X$ such that $\phi(X) \in \pi(\L^2)(B_{\mathcal{H}_{\pi}}(\xi, \delta))$ set $$n(X):=\inf \{n \in \mathbb{N} | \phi(X) \in \pi(\lambda_n)(B_{\mathcal{H}_{\pi}}(\xi, \delta))\}.$$ Then the map $h_{\delta}: X \mapsto \lambda_{n(X)}^{-1}$ is well-defined $\nu$-almost everywhere and Borel. Now, for all $g \in G$ and $\nu$-almost all $X$, $h_{\delta}$ is defined at both $X$ and $g^{-1}X$. Thus, $\pi(h_{\delta}(X))( \phi(X) )\in B_{\mathcal{H}_{\pi}}(\xi, \delta)$ and $\pi(h_{\delta}(g^{-1}X)) (\phi(g^{-1}X)) \in B_{\mathcal{H}_{\pi}}(\xi, \delta)$. But 
  $$\pi(\alpha_s(g^{-1},X)^{-1})(\phi(g^{-1}X)) = \phi(X),$$
  so 
  $$\pi(\alpha_{f_{\delta}}(g^{-1},X)^{-1})(\pi(h_{\delta}(g^{-1}X))(\phi(g^{-1}X))) = \pi(h_{\delta}(X))(\phi(X))$$
  where $f_{\delta}(X):= s(X)h_{\delta}(X)^{-1}$. In other words, for all $g\in G$ and $\nu$-almost all $X$, $$\pi(\alpha_{f_{\delta}}(g^{-1},X)^{-1})(B_{\mathcal{H}_{\pi}}(\xi, \delta)) \cap B_{\mathcal{H}_{\pi}}(\xi, \delta) \neq \emptyset.$$ This implies $\pi(\alpha_{f_{\delta}}(g^{-1},X)^{-1}) \in V_{\delta}$. 
  
   Now, according to Proposition \ref{Proposition: Ruzsa's covering Lemma for cocycles} applied to $\alpha_{f_{\delta}}$ and the constant map $\Delta: X \mapsto V_{\delta}(\xi)$ we find that there is $F \subset \L$ finite such that $\pi(\L) \subset V_{2\delta}(\xi)\pi(F)$. So $\L(2\delta,\xi):= \L^2 \cap \pi^{-1}(V_{2\delta}(\xi))$ and $\xi$ are as in Lemma \ref{Lemma: Compactness criterion for the unitary image of an approximate subgroup} (1) and, hence, $\Lambda$ has property $(T)$. 
 \end{proof}

 \subsection{Finite Generation}
 It is well-known that property (T) for groups implies finite generation (see e.g. \cite{MR2415834}). We will show in the same spirit that: 
 \begin{proposition}\label{Proposition: Approximate subgroups with property T are finitely generated}
  If an approximate subgroup $\L$ of some discrete group has property (T) then the subgroup $\langle \Lambda \rangle$ it generates is finitely generated. More precisely, if $(Q,\epsilon)$ is any Kazhdan pair, then  $\L$ is covered by finitely many left translates of the subgroup $\Delta$ generated by $Q$.
 \end{proposition}
 
 \begin{proof}
  Let $(Q,\epsilon)$ be a Kazhdan pair and let $\Delta$ denote the subgroup generated by $Q$. Then the indicator function $\mathds{1}_{\Delta}$ is a $(Q,\epsilon)$-invariant vector of the quasi-regular representation $(\pi, L^2(\langle \Lambda \rangle/\Delta))$. So, by Lemma \ref{Lemma: Compactness criterion for the unitary image of an approximate subgroup}, we can find $\phi \in L^2(\langle \Lambda \rangle/\Delta)$ with norm $1$ and $(\L(\delta, \phi))_{\delta > 0}$ a family of approximate subgroups contained in $\L^2$ and commensurable to $\L$ such that $||\pi(\lambda)(\phi)-\phi|| < \delta$ for all $\delta >0$ and $\lambda \in \L(\delta,\phi)$. Now let $p: \langle \Lambda \rangle \rightarrow \langle \Lambda \rangle/\Delta$ denote the natural projection. Take $\gamma \in \langle \Lambda \rangle$ such that $\phi(p(\gamma))=\alpha > 0$. So for all $\lambda \in \L(\alpha/2,\phi)$ we have 
  $| \phi(p(\lambda^{-1}\gamma))- \phi(p(\gamma))| \leq ||\pi(\lambda)(\phi)-\phi ||< \alpha/2,$
  meaning $p\left(\lambda^{-1}\gamma\right) \in \phi^{-1}([\alpha/2; +\infty))$. Since $\phi^{-1}([\alpha/2; +\infty))$ is finite, we can find a finite set $F$ of representatives of $\phi^{-1}([\alpha/2; +\infty))$ in $\langle \Lambda \rangle$. Then $\lambda^{-1} \gamma \Delta \cap F \Delta \neq \emptyset$ and $\L(\alpha/2,\phi)$ is contained in $F\Delta\gamma^{-1}$. But $\gamma^{-1}\Lambda(\alpha/2,\phi)\gamma$ is commensurable with $\Lambda$. So there is a finite subset $F' \subset \langle \Lambda \rangle$ such that $\L \subset F'\gamma^{-1}\L(\alpha/2, \phi)\gamma$. Thus, 
  $$\L \subset F'\gamma^{-1}\L(\alpha/2, \phi)\gamma  \subset F'\gamma^{-1}F\Delta.$$ 
 \end{proof}

 As a corollary, we prove Theorem \ref{Theorem: Approximate lattices in property T groups are finitely generated}:

 \begin{proof}[Proof of Theorem \ref{Theorem: Approximate lattices in property T groups are finitely generated}.]
  According to Proposition \ref{Proposition: Heredity of property (T)} we know that $\L$ has property (T). So $\langle \Lambda \rangle$ is finitely generated as a consequence of Proposition \ref{Proposition: Approximate subgroups with property T are finitely generated}.
 \end{proof}

 \section{Superrigidity and Arithmeticity}\label{Section: Superrigidity and Arithmeticity}
 From now on let $A$ be a finite set, let $(k_{\alpha})_{\alpha\in A}$ be a family of local fields of characteristic $0$ and let $(\G_{\alpha})_{\alpha \in A}$ be a family of almost simple algebraic groups defined over $k_{\alpha}$ with $k_{\alpha}$-rank $\geq 2$. We will moreover suppose that the $\G_{\alpha}$'s are \emph{absolutely} almost simple. For any subset $B \subset A$ set $G_B:=\prod_{\alpha \in B} \G_{\alpha}(k_\alpha)$ and let $p_B: G_A \rightarrow G_B$ denote the natural map. Moreover, for all $\alpha \in A$ let $G_{\alpha}$ denote $\G_{\alpha}(k_{\alpha})$ and let $p_{\alpha}$ denote $p_{\{\alpha\}}$.
 
 \subsection{Superrigidity in Bounded Dimension}
  For any local field $l$ of characteristic $0$ we let $p(l)$ denote the unique element of $\mathcal{P}\cup \{\infty\}$ (the set of prime numbers together with $\{\infty\}$) such that $l$ is a finite extension of the $p(l)$-adic field $\Q_{p(l)}$ (where $\Q_{\infty}:=\R$). 
 
 \begin{proposition}\label{Proposition: Partial superrigidity result}
  Let $\L$ be a $\star$-approximate lattice in $G_A$ and $T: \langle \Lambda \rangle \rightarrow \mathbb{H}(l)$ be a group homomorphism towards the $l$-points of an affine $l$-group $\mathbb{H}$. We have:
  \begin{enumerate}
   \item if $p(k_{\alpha})\neq p(l)$ for all $\alpha \in A$, then $T(\L)$ is a relatively compact subset of $\mathbb{H}(l)$;
   \item if $\mathbb{H}$ is an absolutely almost simple algebraic group defined over $l$, $T(\L)$ is not relatively compact in $\mathbb{H}(l)$ and for every $\alpha \in A$ such that $p(k_{\alpha})=p(l)$, $\dim(\G_{\alpha}) \geq \dim(\mathbb{H})$, then there is a continuous group homomorphism $\pi: G_A \rightarrow \mathbb{H}(l)$ that almost extends $T$.  
  \end{enumerate}
 \end{proposition}

 \begin{proof}
 Let $\nu$ be an ergodic proper $G$-invariant Borel probability measure on $\Omega_{\L}^{ext}$. Let $s$ be a Borel section of $\Omega_{\L}^{ext}$ and $\alpha_s(g,X):=s(gX)^{-1}gs(X)$ for all $g\in G$ and $\nu$-almost all $X \in \Omega_{\L}^{ext}$. Consider the Borel cocycle $T\circ \alpha_s :G_A \times \Omega_{\L}^{ext} \rightarrow \mathbb{H}(l)$. Let $\mathbb{L}$ be the algebraic hull of $T\circ \alpha_s$ (\cite[Prop. 9.2.1]{zimmer2013ergodic}) and let $\beta : G_A \times \Omega_{\L}^{ext} \rightarrow \mathbb{H}(l)$ be a Borel cocycle cohomologous to $T\circ \alpha_s$ that takes values in $\mathbb{L}(l)$. Let $\mathbb{F} \ltimes \mathbb{U}$ be a Levi decomposition of $\mathbb{L}$ with $\mathbb{F}$ reductive and $\mathbb{U}$ unipotent. Let $p: \mathbb{L}(l) \rightarrow \mathbb{F}(l)$ be the natural map. Then the algebraic hull of $ p \circ \beta:G_A \times \Omega_{\L}^{ext} \rightarrow \mathbb{F}(l)$ is $\mathbb{F}$. Indeed, otherwise there would exist a proper $l$-subgroup $\mathbb{F}'$ and a Borel map $\psi: \Omega_{\L}^{ext} \rightarrow \mathbb{F}(l)$ such that for all $g \in G_A$ and $\nu$-almost every $X \in \Omega_{\L}^{ext}$ we have $\psi(gX)\left(p \circ \beta\right)(g,X) \psi(X)^{-1} \in \mathbb{F'}(l)$. Taking a Borel map $\tilde{\psi}: \Omega_{\L}^{ext} \rightarrow \mathbb{L}(l)$ such that $p \circ \tilde{\psi}=\psi$ we would have for all $g \in G_A$ and $\nu$-almost every $X \in \Omega_{\L}^{ext}$ that
 $$\tilde{\psi}(gX)\beta(g,X) \tilde{\psi}(X)^{-1} \in \left(\mathbb{F'}\ltimes \mathbb{U}\right)(l) \subsetneq \mathbb{L}(l).$$
 A contradiction. Thus, the cocycle $p\circ \beta$ has a reductive algebraic hull. By \cite[Th. 3.16]{MR2039990} there are a continuous group homomorphism $\pi:G_A \rightarrow \mathbb{F}(l)$ and a cocycle $z:G_A \times \Omega_{\L}^{ext} \rightarrow \mathbb{F}(l)$ that takes values in a compact subgroup centralising $\pi(G_A)$ such that $p \circ \beta$ is cohomologous to the cocycle defined $\nu$-almost everywhere by $g,X \mapsto \pi(g)z(g,X)$.

 Suppose first that $p(k_{\alpha})\neq p(l)$ for all $\alpha \in A$. Then $\pi$ is trivial according to \cite[I.2.6.1, (i)]{MR1090825}. So $p\circ \beta$ is cohomologous to a cocycle that takes values in a compact subgroup of $\mathbb{F}(l)$. Reasoning as above we see that $\beta$ is cohomologous to a cocycle that takes values in an amenable subgroup. By \cite[Th. 9.1.1]{zimmer2013ergodic} we thus have that $\beta$ is cohomologous to a cocycle that takes values in a compact subgroup of $\mathbb{H}(l)$. Whence $T \circ \alpha_s$ is cohomologous to a cocycle that takes values in a compact subgroup of $\mathbb{H}(l)$. By Corollary \ref{Corollary: Compact cocycles} the subset $T(\L)$ is relatively compact in $\mathbb{H}(l)$. 
 
 Suppose now that the assumptions of (2) are satisfied. If $\pi$ is trivial, then as above we conclude that $T(\L)$ is relatively compact in $\mathbb{H}(l)$. Otherwise according to \cite[I.2.6.2]{MR1090825} the Zariski closure of $\pi(G_A)$ is semi-simple. Moreover, one can see applying \cite[I.2.6.1,(iii)]{MR1090825} again that $\dim(\G_\alpha) \leq \dim(\mathbb{F})$ for some $\alpha \in A$ with $p(k_{\alpha})=p(l)$. As a consequence, we have the equality $\dim(\mathbb{F})=\dim(\mathbb{H})$ and this yields $\mathbb{H}=\mathbb{L}=\mathbb{F}$ since $\mathbb{H}$ is connected. So $p=\id$ and $p \circ \beta = \beta$. So $\pi$ almost extends $T$ according to Proposition \ref{Proposition: Constant cocycles in semi-simple groups}.  
\end{proof}
 
 \subsection{Compact Finiteness}\label{Subsection: Rigidity of compact approximate subgroups} We prove now a general finiteness property of compact images of approximate groups.
 \begin{proposition}\label{Proposition: Rigidity of compact approximate subgroups}
  Let $\L$ be an approximate subgroup of some discrete group. Take a family $\mathcal{F}$ of group homomorphisms $\tau: \langle \Lambda \rangle \rightarrow H_{\tau}$ with $H_{\tau}$ a locally compact group, $\tau(\langle \Lambda \rangle)$ dense and $\overline{\tau(\L)}$ compact with non-empty interior. Suppose moreover that $H_{\tau}$ contains a closed non-discrete non-compact topologically simple subgroup $S_{\tau}$ contained in all non-trivial normal subgroups. Then there is $\mathcal{F}' \subset \mathcal{F}$ finite such that:
  \begin{enumerate}
   \item for any group homomorphism $\sigma: \langle \Lambda \rangle \rightarrow G \in \mathcal{F}$ there is $\tau: \langle \Lambda \rangle \rightarrow H \in \mathcal{F}'$ and a continuous group homomorphism $\psi: H \rightarrow G$ such that $\psi \circ \tau = \sigma$;
   \item the intersection $\overline{\tau_{\mathcal{F}'}(\L^2)}\cap \prod_{\tau \in \mathcal{F}'}S_{\tau}$ is a neighbourhood of the identity in $\prod_{\tau \in \mathcal{F}'}S_{\tau}$ where $\tau_{\mathcal{F}'}:=\prod_{\tau \in \mathcal{F}'} \tau$.
  \end{enumerate}
      
 \end{proposition}
 
 \begin{proof}
  For any $X \subset \mathcal{F}$  write $H_X:=\prod_{\tau \in X} H_{\tau}$, $\tau_X:=\prod_{\tau \in X} \tau : \langle \Lambda \rangle \rightarrow H_X$ the diagonal map and $p_{X}: H_{\mathcal{F}} \rightarrow H_X$ the natural projection. Since $\overline{\tau_{\mathcal{F}}(\L)} \subset H_{\mathcal{F}}$ is a compact approximate subgroup, there is a topology on the subgroup $L$ it generates finer than the induced topology, with $L$ locally compact and $\overline{\tau_{\mathcal{F}}(\L^2)}=:V$ a neighbourhood of the identity (see \cite[Th. 4.1]{machado2019goodmodels}). From now on, we will consider $L$ equipped with this topology. Note that the restriction to $L$ of the natural projection $H_{\mathcal{F}} \rightarrow H_X$ is continuous with respect to that topology. Furthermore, for every $\tau \in \mathcal{F}$ one sees that $L$ projects surjectively to $H_{\tau}$ - indeed, the projection contains $\overline{\tau(\Lambda)}$, which has non-empty interior, and $\tau(\langle \Lambda \rangle)$, which is dense. 
  
  Write $N_{\tau} \leq L$ the kernel of $L \rightarrow H_{\tau}$. We claim that there is a finite subset $X \subset \mathcal{F}$ such that $\bigcap_{\tau \in X} N_{\tau} =\{e\}$. Indeed, choose a symmetric relatively compact open neighbourhood $W$ of the identity in $L$. Write $K:=\overline{VW^2V}\setminus W$. Since both $V$ and $W$ are relatively compact and $W$ is open, $K$ is compact. Moreover, $K \cap \bigcap_{\tau \in \mathcal{F}} N_{\tau}  = \emptyset$ so $K \cap \bigcap_{\tau \in X} N_{\tau}  = \emptyset$ for some finite subset $X \subset \mathcal{F}$. Let us consider the subset $C:=W \cap \bigcap_{\tau \in X} N_{\tau} $. Note first that $C$ is a symmetric subset that is open in $\bigcap_{\tau \in X} N_{\tau}$. Take $g,h \in C$, then 
  $$gh \in W^2 \cap \bigcap_{\tau \in X} N_{\tau} \subset \left( W \cup \left(W^2\setminus W\right)\right) \cap \bigcap_{\tau \in X} N_{\tau} \subset \left(W \cup K\right) \cap \bigcap_{\tau \in X} N_{\tau} =  C,$$
  where the last inequality is a consequence of $K\cap \bigcap_{\tau \in X} N_{\tau}  = \emptyset$. We have thus shown that $C$ is a subgroup. Similarly, if $g \in C$ and $h \in V$, then 
  $$hgh^{-1} \in \left( W \cup \left(VWV \setminus W\right)\right) \cap \bigcap_{\tau \in X} N_{\tau} \subset \left( W \cup K\right) \cap \bigcap_{\tau \in X} N_{\tau} = C.$$
  So $C$ is normalised by $V$ and, hence, by $\langle V \rangle = L$. Now, $C$ is an open subgroup of $\bigcap_{\tau \in X} N_{\tau}$ contained in $W$, so it is compact in $L$. Therefore, for all $\tau \in \mathcal{F}$ the projection $C_{\tau}$ of $C$ to $H_{\tau}$ is a compact normal subgroup. If it were non-trivial, $C_{\tau}$ would contain $S_{\tau}$ which would contradict its compactness. So $C_{\tau}=\{e\}$ and, thus, $C=\{e\}$. In other words, the closed normal subgroup $\bigcap_{\tau \in X} N_{\tau}$ is discrete in $L$ which is compactly generated. So $\bigcap_{\tau \in X} N_{\tau}$ is countable. For all $\tau \in \mathcal{F}$, if the projection of $\bigcap_{\tau \in X} N_{\tau}$ to $H_{\tau}$ were non-trivial, then it would contain $S_{\tau}$ and would be uncountable. So the projection of $\bigcap_{\tau \in X} N_{\tau}$ to $H_{\tau}$ is trivial for all $\tau \in \mathcal{F}$. So $\bigcap_{\tau \in X} N_{\tau}=\{e\}$ and the claim is proved.

  Take $X \subset \mathcal{F}$ of minimal cardinality with $\bigcap_{\tau \in X} N_{\tau} = \{e\}$ and write $L_{X}:=p_{X}(L)$. By minimality of $X$ we know that $L_{X} \cap H_{\tau}$ is non-trivial for all $\tau \in X$. Since $L$ projects surjectively to $H_{\tau}$, we have that $L_{X} \cap H_{\tau}$ is normal in $H_{\tau}$. So $L_{X} \cap H_{\tau}$ contains $S_{\tau}$. As a consequence, $\prod_{\tau \in X} S_{\tau} \subset L_{X}$. Now, $\overline{\tau_X(\Lambda^2)}$ is a neighbourhood of the identity in $L_X$, so $\overline{\tau_X(\Lambda^2)} \cap \prod_{\tau \in X} S_{\tau} $ is a neighbourhood of the identity in $\prod_{\tau \in X} S_{\tau}$ and, hence,  $X$ satisfies (2). Take now $\sigma \notin X$. Since the projection of $p_{X\cup \{\sigma\}}(L)$ to $L_X$ has trivial kernel, by Goursat's lemma $p_{X\cup \{\sigma\}}(L)$ must be the graph of a group homomorphism $\pi: L_{X} \rightarrow H_{\sigma}$. Since $p_{X\cup \{\sigma\}}(L)$ is Borel, $\pi$ is continuous (\cite[App. A]{zimmer2013ergodic}) where we consider on $L_X$ the topology inherited from $L$. Furthermore, by construction, $\pi$ satisfies $\pi \circ \tau_X = \sigma$ and $\pi(L_X)=H_{\sigma}$. We wish to show now that $\pi$ factors through one of the projections $p_{\tau}: L_X \rightarrow H_{\tau}$ for some $\tau \in X$. Take $\tau_1,\tau_2 \in X$ such that $\pi(H_{\tau_1} \cap L_X)$ and $\pi(H_{\tau_2} \cap L_X)$ are non-trivial. For $i =1,2$, $\pi(H_{\tau_i} \cap L_X)$ is normal in $H_{\sigma}$. So it contains $S_{\sigma}$. If $\tau_1\neq\tau_2$, then $\pi(H_{\tau_1} \cap L_X)$ and $\pi(H_{\tau_2} \cap L_X)$ commute. In particular, $S_{\sigma}$ is abelian. This contradicts that $S_{\sigma}$ is topologically simple and infinite. So $\tau_1=\tau_2$. As a consequence, there is $\tau_{\sigma}\in X$ such that for all $\tau  \in X \setminus \{\tau_{\sigma}\}$ we have $\pi(H_{\tau} \cap L_X)=\{e\}$. Since $p_{\tau_{\sigma}}(L_X)=H_{\tau_\sigma}$, $\pi$ thus factors through $p_{\tau_{\sigma}}$ and yields a continuous group isomorphism $\psi: H_{\tau_{\sigma}} \rightarrow H_{\sigma}$ such that $\pi = \psi \circ p_{\tau_{\sigma}}$. Therefore, we have 
  $$\sigma = \pi \circ \tau_X = \psi \circ p_{\tau_{\sigma}} \circ \tau_X = \psi \circ \tau_{\sigma}. $$
   So (1) is satisfied as well and $\mathcal{F}':=X$ works. 
 \end{proof}

Let $k$ be a local field of characteristic $0$ and $\mathbb{H}$ be an absolutely simple (centreless) algebraic group defined over $k$. A fundamental result due to Borel and Tits (see \cite[\S I.1.5 and \S I.2.3]{MR1090825} and \cite{BoTits} for this and more) asserts that the subgroup $\mathbb{H}(k)^+$ generated by all the unipotent elements in $\mathbb{H}(k)$ is a closed normal simple subgroup and every non-trivial normal subgroup of $\mathbb{H}(k)$ contains $\mathbb{H}(k)^+$ (In fact, $\mathbb{H}(k)^+$ is contained in any non-trivial normal subgroup of any non-compact open subgroup of $\mathbb{H}(k)$). Since $k$ has characteristic $0$, the subgroup $\mathbb{H}(k)^+$ is moreover open, has finite index and is the minimal non-compact open subgroup of $\mathbb{H}(k)$ (\cite[9.10]{BoTits}). Along with Lemma \ref{Lemma: Compact Zariski-dense approximate subgroups of simple groups} below, this fact will allow us to use Proposition \ref{Proposition: Rigidity of compact approximate subgroups} in the proof of Theorem \ref{Main theorem}.  

\begin{lemma}\label{Lemma: Compact Zariski-dense approximate subgroups of simple groups}
Let $k$ be a local field of characteristic $0$ and let $\mathbb{G}$ be an absolutely simple (centreless) algebraic group defined over $k$.  Let $\Lambda$ be an infinite compact approximate subgroup of $G:=\mathbb{G}(k)$ such that $\langle \Lambda \rangle$ is unbounded and Zariski-dense. Then there are a closed subfield $k' \subset k$ and an absolutely almost simple algebraic group $\mathbb{L}$ defined over $k'$ such that: 
\begin{enumerate}
\item $\mathbb{L}(k)=\mathbb{G}(k)$;
\item $\Lambda^2$ is a neighbourhood of the identity in $\mathbb{L}(k')$. 
\end{enumerate}
In addition, $\Tr \Ad \Lambda \subset k'$. 
\end{lemma}

Obtaining information regarding traces of adjoints is often a key step towards arithmeticity results. Here as well, the values of these traces will play a crucial role. We refer to \cite[\S I.1.4]{MR1090825}, \cite[Lem. 6.1.6]{zimmer2013ergodic} and \cite{MR0279206} and references therein for background and more.

\begin{proof}
We know that $k$ is a finite extension of a local subfield $k_0$ isomorphic to either $\mathbb{R}$ or $\mathbb{Q}_p$ for some prime number $p$. Write $\mathbb{H}:=R_{k/k_0} \mathbb{G}$ Weil's restriction of scalars of $\mathbb{G}$ and recall that $\mathbb{H}$ is a simple algebraic group defined over $k_0$ such that $\mathbb{H}(k_0)=G$ (see \cite[\S I.1.7]{MR1090825}). We will exploit the following observation: let $L \subset G (=\mathbb{H}(k_0))$ be a Zariski-closed subgroup for the $k_0$-structure and suppose that $L$ is Zariski-dense in the $k$-structure (i.e. in $\mathbb{G}(k)$), then the Zariski-connected component of $L$ is simple. Indeed, if $R$ denotes the radical of $L$, then the Zariski-closure $\overline{R}^k$ of $R$ in the $k$-structure (i.e. in $\mathbb{G}(k)$) is a Zariski-closed soluble subgroup of $G(=\mathbb{G}(k))$. It is moreover normalised by $L$, so $\overline{R}^k$ is normal in $G$. By assumption, $\overline{R}^k$ is trivial. So $R$ is trivial. Hence, the connected component of the identity $L$ is semi-simple. Similar arguments show that it has at most one simple factor and is centreless. 

Now consider $\Lambda$ in $G$ seen as a $k_0$-Lie group. Let $\mathfrak{g}$ denote the $k_0$-Lie algebra of $G$. According to \cite[Th. 1.4]{machado2019goodmodels}, we can associate with $\Lambda$ a non-trivial Lie subalgebra $\mathfrak{l}$ of $\mathfrak{g}$ stable under $\Ad(\langle \Lambda \rangle)$. The stabiliser $L$ of $\mathfrak{l}$ is a subgroup of $G$ that is Zariski-closed in the $k_0$-structure. Since it contains $\langle \Lambda \rangle$, $L$ is also Zariski-dense in the $k$-structure. By the above paragraph, its Zariski-connected component $L^0$ is simple. But the $k_0$-Lie subalgebra of $L$ contains $\mathfrak{l}$ as an ideal. By simplicity, $\mathfrak{l}$ is therefore the Lie subalgebra of $L$. According to \cite[Th. 1.4]{machado2019goodmodels}, we find that $L^0 \cap \Lambda^2$ is a neighbourhood of the identity in $L^0$ and, hence, $L$ (with the Hausdorff topology). 

By \cite[\S I.1.7]{MR1090825} now, there exists a local field $k'$ and an absolutely simple algebraic group $\mathbb{L}$ defined over $k'$  such that $L^0= \mathbb{L}(k')$. Since $L^0$ is non-compact - because it is a finite index subgroup of $L$ which contains the unbounded subgroup $\langle \Lambda \rangle$ - and both $L^0$ and $\mathbb{G}(k)$ are centreless, we can apply a result of Borel and Tits (see \cite[Thm. I.1.8.1]{MR1090825}) which allows us to assume that we have chosen $k'$ and $\mathbb{L}$ such that $k'$ is a closed subfield of $k$ and $\mathbb{L}=\mathbb{G}$ as algebraic groups over $k$. Moreover, the normaliser of $\mathbb{L}(k')$ in $\mathbb{L}(k)$ is $\mathbb{L}(k')$ itself (e.g. \cite[Lem. VII.6.2]{MR1090825}). And so $\Lambda \subset \mathbb{L}(k')$. In particular, $\Tr \Ad \Lambda \subset k'$  (by e.g. \cite[I.1.4.8]{MR1090825}, see also the elementary explanation in the proof of \cite[Lem. 6.1.6]{zimmer2013ergodic} in the case $k_0=\mathbb{R}$).
\end{proof}

 \subsection{Arithmeticity}\label{Subsection: Arithmeticity : Semi-Simple Algebraic Groups}  We will proceed by induction on the cardinality of $A$. To do so we will need the following lemma about reduction of $\star$-approximate lattices:
  
  \begin{lemma}\label{Lemma: Reduction of star-approximate lattices}
   Let $\L$ be a $\star$-approximate lattice in $G_A$:
   \begin{enumerate}
    \item take $\gamma \in \langle \Lambda \rangle$ non-central and $B_{\gamma}:=\{\alpha \in A | p_{\alpha}(\gamma) \text{ non-central}\}$, then $p_{B_{\gamma}}(\L)$ is a $\star$-approximate lattice;
    \item take $B \subset A$, if $p_B(\L)$ is a $\star$-approximate lattice, then $p_{A\setminus B}(\L)$ is a $\star$-approximate lattice.
   \end{enumerate}
  \end{lemma}
  
  \begin{proof}
   First of all, for $g \in G_A$ and $f$ centralising $G_{A\setminus B_{\gamma}}$ we have $gf g^{-1} = p_{B_{\gamma}}(g)f p_{B_{\gamma}}(g)^{-1}$. Let $N$ be the normal subgroup of $\langle \Lambda \rangle$ generated by $\gamma$. Since $\langle \Lambda \rangle$ has property (S) (\cite[Prop. 6.5]{machado2019goodmodels}), the Borel density theorem yields $C_{G_A}(N)=G_{A\setminus B_{\gamma}} \times Z_{B_{\gamma}}$ where $C_{G_A}(\cdot)$ is the centraliser in $G_A$ and $Z_{B_{\gamma}}$ is the centre of $G_{B_{\gamma}}$. There is thus $F \subset N$ finite such that $C_{G_A}(F)=G_{A\setminus B_{\gamma}} \times Z_{B_{\gamma}}$. Take $n \geq 0$ such that $F \subset \L^n$. There is a neighbourhood $U \subset G_{B_{\gamma}}$ of $\{e\}$ such that $\L^{n+4} \cap Uf U^{-1}=\{f\}$ for all $f \in F$. So $\lambda f \lambda^{-1} = f$ if $f \in F$ and $\lambda \in \L^2 \cap p_{B_{\gamma}}^{-1}(U)$. Hence, $p_{B_{\gamma}}(\L^2) \cap U \subset Z_{B_{\gamma}}$ so $p_{B_{\gamma}}(\L)$ is uniformly discrete. Now (1) follows from Corollary \ref{Corollary: Intersection star-approximate lattice and closed subgroup}. If $p_B(\L)$ is locally finite, then $\L^2 \cap G_{A\setminus B}$ is a $\star$-approximate lattice in $G_{A\setminus B}$ by Corollary \ref{Corollary: Intersection star-approximate lattice and closed subgroup}. We can thus find $\gamma \in \L^2 \cap G_{A\setminus B}$ such that $B_{\gamma}= A \setminus B$ (e.g. Corollary \ref{Corollary: Star-approximate lattices are bi-syndetic}). So (2) follows from (1). 
  \end{proof}

  \begin{proof}[Proof of Theorem \ref{Main theorem}.]
  First of all, using Weil's restriction of scalars there is for every $\alpha$ a finite extension $k_{\alpha}'$ of $k_{\alpha}$ and an absolutely almost simple $k_{\alpha}'$ group $\mathbb{G}_{\alpha}'$ of $k_\alpha'$-rank at least $2$ such that $\mathbb{G}_{\alpha}(k_{\alpha})=\mathbb{G}_{\alpha}'(k_{\alpha}')$ (\cite[\S I.1.7]{MR1090825}). So we can assume that the groups $\mathbb{G}_{\alpha}$ are absolutely almost simple. The crux of the proof is to establish the following claim:
   \begin{claim}\label{Claim: Discrete counterpart}
    There is $H$ a finite product of groups of points of simple centreless algebraic groups over local fields of characteristic $0$ and a group homomorphism $\tau: \langle \Lambda \rangle \rightarrow H$ such that $\tau(\L)$ is relatively compact and topologically generates an open finite index subgroup, and the image of the diagonal map $\langle \Lambda \rangle \rightarrow G_A \times H$ is discrete.  
   \end{claim}
   
   Let us explain why the claim is sufficient. Since $\L$ is a $\star$-approximate lattice and is contained in the projection $M$ to $G_A$ of $\id \times \tau(\langle \Lambda \rangle) \cap G_A \times W_0$ for some compact neighbourhood of the identity $W_0 \subset H$, the subgroup $\Gamma_{\tau}:=\id \times \tau(\langle \Lambda \rangle)$ is a lattice (Lemma \ref{Lemma: A weak model set with finite co-volume is a model set}). So $\L$ is contained in and commensurable to the model set $M$ by Corollary \ref{Corollary: Towers of star-approximate subsets are commensurable}.
   
   Furthermore, we only need to establish the claim when the $\mathbb{G}_{\alpha}$'s are centreless. Indeed, for $\alpha \in A$ write $\mathbb{G}_{\alpha}^{ad}$ the adjoint of $\mathbb{G}_{\alpha}$. Then the natural map $\pi: G_A \rightarrow G_A^{ad}:=\prod_{\alpha \in A}\mathbb{G}_{\alpha}^{ad}(k_{\alpha})$ is a continuous group homomorphism with finite kernel and finite index range. Therefore, $\pi(\Lambda)$ is a $\star$-approximate lattice in the group $G_A^{ad}$ (Corollary \ref{Corollary: Intersection star-approximate lattice and closed subgroup}). Suppose we can find $\tau: \langle \pi(\Lambda) \rangle \rightarrow H$ as in Claim \ref{Claim: Discrete counterpart} applied to $\pi(\Lambda)$. Then $\tau \circ \pi (\Lambda)$ is relatively compact and topologically generates an open finite index subgroup. Moreover, the graph $\Gamma_{\tau \circ \pi}$ of $\tau \circ \pi$ is equal to $(\pi\times \id)^{-1}\left(\id \times \tau(\langle \pi(\Lambda) \rangle) \right)$ which must be discrete because $\pi$ has finite kernel and $\id \times \tau(\langle \pi(\Lambda) \rangle)$ is discrete. From now on, we will therefore make the assumption that the $\mathbb{G}_{\alpha}$'s are centreless.
   \medbreak
  We will prove the claim by induction on $|A|$. If $|A|=0$ the result is obvious. Suppose now that $|A| \geq 1$. If there is a proper non-trivial subset $B \subset A$ such that $p_B(\L)$ is a $\star$-approximate lattice, then $p_{A\setminus B}(\L)$ is a $\star$-approximate lattice as well by Lemma \ref{Lemma: Reduction of star-approximate lattices}. Applying Claim \ref{Claim: Discrete counterpart} to both $p_B(\L)$ and $p_{A\setminus B}(\L)$ and taking the Cartesian product gives $H$ and $\tau$ as in Claim \ref{Claim: Discrete counterpart} except that $\tau(\L)$ might not topologically generate a finite index open subgroup of $H:=\prod_{i\in I} H_i$ where the $H_i$'s are the simple factors of $H$. However, at the very least the projection of $\tau(\L)$ to any simple factor $H_i$ topologically generates an open finite index subgroup of $H_i$ according to the induction hypothesis. Thanks to part (2) of Proposition \ref{Proposition: Rigidity of compact approximate subgroups} we will now be able to conclude. Indeed, for all $i \in I$ let $\tau_i: \langle \Lambda \rangle \rightarrow H_i$ be the composition of the map $\tau$ with the natural projection $H \rightarrow H_i$. Define $\mathcal{F}:=\{\tau_i: \langle \Lambda \rangle \rightarrow \overline{\langle \Lambda \rangle} | i \in I\}$. Notice that for all $i \in I$, since $H_{\tau_i}:=\overline{\langle \Lambda \rangle}$ is a finite index subgroup of $H_i$, the subgroup $S_{\tau_i}:=H_i^+$ generated by the unipotent elements of $H_i$ is contained in all non-trivial normal subgroups of $H_{\tau_i}$ (see the paragraph below the proof of Proposition \ref{Proposition: Rigidity of compact approximate subgroups}). So let $\mathcal{F}' \subset \mathcal{F}$ be given by Proposition \ref{Proposition: Rigidity of compact approximate subgroups} applied to $\mathcal{F}$ and let $J \subset I$ be the subset of indices of elements in $\mathcal{F}'$.  Write $p_J: H \rightarrow \prod_{j \in J} H_j$ the natural projection. Then $p_J \circ \tau$ satisfies all the conditions of Claim  \ref{Claim: Discrete counterpart}. In particular, $\overline{p_J \circ \tau(\langle \Lambda \rangle)}$ contains the finite index open subgroup $\prod_{j \in J} H_j^+$ by (2) of Proposition \ref{Proposition: Rigidity of compact approximate subgroups}; by (1) there is a continuous group homomorphism $\psi: \prod_{j \in J} H_j \rightarrow H$ such that $\tau = \psi \circ p_J \circ \tau$. Since the graph of $\tau$ is discrete, so is the graph of $p_J \circ \tau$.
   
  Suppose otherwise that there is no such $B$. According to part (1) of Lemma \ref{Lemma: Reduction of star-approximate lattices} the projection of $\langle \Lambda \rangle$ to any factor is injective. It has property (S) by \cite[Prop. 6.5]{machado2019goodmodels} and, thus, is Zariski-dense according to the Borel density theorem. Choose $\alpha$ with $\dim( G_{\alpha})$ minimal. According to Theorem \ref{Theorem: Approximate lattices in property T groups are finitely generated} the group $\langle \Lambda \rangle$ is finitely generated. Therefore the set $\{\Tr \Ad p_{\alpha} \gamma | \gamma \in \langle \Lambda \rangle\}$ generates a finitely generated field $K$. According to \cite{MR0279206} and \cite[\S I.1.7]{MR1090825} we can identity $G_{\alpha}$ with the $k_{\alpha}$-points of an absolutely simple algebraic group $\mathbb{H}$ defined over $K$ such that $p_{\alpha}(\langle \Lambda \rangle)\subset \mathbb{H}(K)$ (recall that $\G_{\alpha}$ is centreless, and so is $\mathbb{H}$). Let $\mathcal{F}$ be the family $\{ \widehat{\sigma}| p_{\alpha}(\langle \Lambda \rangle) \rightarrow \mathbb{H}^{\sigma}(k) \}_{\sigma}$ where $\sigma: K \rightarrow k$ runs through the classes of field embeddings with $k$ local, $\sigma(K)$ dense and such that the natural map $\widehat{\sigma}:p_{\alpha}\left(\langle \Lambda \rangle\right) \rightarrow \mathbb{H}^{\sigma}(k)$ has non-compact image but sends $\L$ to a relatively compact set. Here, two embeddings $\sigma_1: K \rightarrow k_1$ and $\sigma_2: K \rightarrow k_2$ are identified if there exists a continuous isomorphism of fields $\phi: k_1 \rightarrow k_2$ such that $\sigma_2 = \phi \circ \sigma_1$. We wish now to apply Proposition \ref{Proposition: Rigidity of compact approximate subgroups}. By the paragraph following Proposition \ref{Proposition: Rigidity of compact approximate subgroups}, for every $\widehat{\sigma}: p_{\alpha}\left(\langle \Lambda \rangle\right) \rightarrow \mathbb{H}^{\sigma}(k)\in \mathcal{F}$, we can choose the subgroup $\mathbb{H}^{\sigma}(k)^+$ generated by the unipotent elements as the subgroup $S_{\widehat{\sigma}}$. To be able to apply Proposition \ref{Proposition: Rigidity of compact approximate subgroups} we will show that $\widehat{\sigma}\circ p_{\alpha}(\Lambda)$ has non-empty interior. First of all, if $\widehat{\sigma} \circ p_{\alpha}(\Lambda)$ is finite, then $p_{\alpha}(\Lambda)$ is finite. Since the restriction of $p_{\alpha}$ to $\Lambda$ is injective, this implies that $\Lambda$ itself is finite. A contradiction. So $\overline{\widehat{\sigma} \circ p_{\alpha}(\Lambda)}$ is infinite and compact. Now $\mathbb{H}$ and, hence, $\mathbb{H}^{\sigma}$ are assumed absolutely simple (thus, centreless). By Lemma \ref{Lemma: Compact Zariski-dense approximate subgroups of simple groups}, if $\overline{\widehat{\sigma} \circ p_{\alpha}(\Lambda)}$ has empty interior, then there is a proper closed subfield $k' \subset k$ such that $\sigma \left( \Tr \Ad p_{\alpha}(\langle \Lambda \rangle)\right) = \Tr \Ad \widehat{\sigma} \circ p_{\alpha}(\langle \Lambda \rangle) \subset k'$. But $\Tr \Ad p_{\alpha}(\langle \Lambda \rangle)$ generates $K$ which has dense image in $k$. Contradicting that $k'$ is a proper subfield of $k$. So $\overline{\widehat{\sigma} \circ p_{\alpha}(\Lambda)}$ indeed has non-empty interior in $\mathbb{H}^{\sigma}(k)$. In particular, $H_{\widehat{\sigma}}:=\overline{\widehat{\sigma} \circ p_{\alpha}(\langle\Lambda\rangle)}$ is an open non-compact subgroup of $\mathbb{H}^{\sigma}(k)$ - so every non-trivial normal subgroup of $H_{\widehat{\sigma}}$ contains $S_{\widehat{\sigma}}=\mathbb{H}^{\sigma}(k)^+$ by the paragraph following Proposition \ref{Proposition: Rigidity of compact approximate subgroups}. Hence, Proposition \ref{Proposition: Rigidity of compact approximate subgroups} applied to $\mathcal{F}$ provides a finite subset $\mathcal{F}' \subset \mathcal{F}$. Write $H:=\prod_{\langle \Lambda \rangle \rightarrow \mathbb{H}^{\sigma}(k) \in \mathcal{F}'} \mathbb{H}^{\sigma}(k)$ and let $\tau: \langle \Lambda \rangle \rightarrow H$ be the natural map. Suppose that $\id \times \tau (\langle \Lambda \rangle) \subset G_A \times H$ is not discrete. There is then an infinite subset $X \subset \langle \Lambda \rangle$ such that $\id \times \tau(X)$ is bounded in $G_A \times H$. But by \cite[Lem. 2.1]{MR2373146} there is $\sigma: K \rightarrow k$ with $k$ local, $\sigma(K)$ dense and $\widehat{\sigma} \circ p_{\alpha}(X)$ unbounded. Now, either $ \widehat{\sigma} \circ p_{\alpha}(\Lambda)$ is unbounded or it is bounded. If $\widehat{\sigma} \circ p_{\alpha}(\Lambda)$ is unbounded, by Proposition \ref{Proposition: Partial superrigidity result}, $\widehat{\sigma} \circ p_{\alpha}$ almost extends to a continuous group homomorphism $\phi: G_A \rightarrow \mathbb{H}^\sigma(k)$. In fact, $\widehat{\sigma} \circ p_{\alpha}$ extends since $\mathbb{H}^\sigma(k)$ is centreless. Thus, $\widehat{\sigma} \circ p_{\alpha}(X)=\phi(X)$. But $X$ is a relatively compact subset of $G_A$, so $\phi(X)$ is a relatively compact subset of $\mathbb{H}^\sigma(k)$, contradicting that $\widehat{\sigma} \circ p_{\alpha}(X)$ is unbounded. If now $\widehat{\sigma} \circ p_{\alpha}(\Lambda)$ is bounded, then by Proposition \ref{Proposition: Rigidity of compact approximate subgroups} we can find a continuous group homomorphism $\pi : H \rightarrow  \mathbb{H}^{\sigma}(k)$ such that $(\widehat{\sigma} \circ p_{\alpha})_{|\langle \Lambda \rangle} = \pi \circ \tau$. By our assumption $\tau(X)$ is a relatively compact subset of $H$, so $\pi(\tau(X))= \widehat{\sigma} \circ p_{\alpha}(X)$ is a relatively compact subset $\mathbb{H}^{\sigma}(k)$. Again, we reach a contradiction. Therefore, there can be no infinite subset $X$ of $\langle \Lambda \rangle$ with $\id \times \tau(X)$ bounded in $G_A \times H$ i.e. $\id \times \tau(\langle \Lambda \rangle)$ is discrete.  
  \end{proof}

  \begin{proof}[Proof of Theorem \ref{First theorem}.]
 We apply Theorem \ref{Main theorem} and get $B$,  $(k_{\beta})_{\beta \in B}$, $k_{\beta}$-groups $(\mathbb{H}_{\beta})_{\beta\in B }$ and a lattice $\Gamma$ in $\mathbb{G}(\mathbf{R}) \times \prod_{\beta \in B} \mathbb{H}_\beta(k_{\beta})$. If we take $B$ minimal, then $\Gamma$ is irreducible. Let $v_0$ denote the usual Archimedean place on $\mathbb{R}$. According to Margulis' arithmeticity theorem \cite[IX.1]{MR1090825}, we can assume that there are a number field $K \subset \mathbb{R}$, pairwise inequivalent places $\{v_{\beta}\}_{\beta \in B}$ of $K$ that are all inequivalent to the restriction of $v_0$ and containing all the Archimedean places inequivalent to $v_0$, a $K$-group $\mathbb{G}'$ with $\mathbb{G}'(\mathbb{R})=\mathbb{G}(\mathbb{R})$ such that: $\langle \Lambda \rangle$ is contained in $\mathbb{G}'(K)$;  $\mathbb{H}_{\beta}(k_{\beta}) = \mathbb{G}'(K_{v_{\beta}})$; and $\Gamma$ is commensurable to $\mathbb{G}'(\mathcal{O}_{K})$ via the diagonal embedding. We may identify $\G'$ with a $K$-subgroup of $\SL_n(\mathcal{O}_{K})$ in such a way that $\mathbb{G}'(\mathcal{O}_{K})$ is $\G'(K) \cap \GL_n(\mathcal{O}_{K})$. For every $\beta \in B$, we can define the subset 
 $$U_{\beta}:=\{g \in \SL_n(K_{v_{\beta}}) | g- I, g^{-1}-I \text{ have entries in } O_{\beta}\}$$
 where $O_{\beta}$ is the (closed) unit ball of $K_{v_{\beta}}$. Then $U_{\beta}$ is a symmetric compact neighbourhood of the identity in $\SL_n(K_{v_{\beta}})$. Write $p_{\beta}: \SL_n(K) \rightarrow \SL_n(K_{v_{\beta}})$ and notice that $\mathbb{G}'(\PVS(K))$ is equal to $\mathbb{G}'(\mathcal{O}_{K}) \cap \bigcap_{\beta \in B} p_{\beta}^{-1}\left( U_{\beta}\right)$ (see Subsection \ref{Subsection: Definition and first properties first part}). By Theorem \ref{Main theorem} now, $p_{\beta}(\Lambda)$ is a relatively compact subset of $\SL_n(K_{v_{\beta}})$ and, thus, is covered by finitely many translates of $U_{\beta}$. Furthermore, if $p_B$ denotes the diagonal map $\mathbb{G}'(\mathcal{O}_{K}) \rightarrow \prod_{\beta \in B} \SL_n(K_{v_{\beta}})$ defined by $\gamma \mapsto (p_{\beta}(\gamma))_{\beta \in B}$, then $p_B(\Lambda)$ is covered by finitely many translates of the neighbourhood $\prod_{\beta \in B} U_{\beta}$. Therefore, $\Lambda$ is covered by finitely many translates of $p_B^{-1}(\prod_{\beta \in B} U_{\beta})$ (see \cite[Lem. 3.1]{machado2019goodmodels} for details). Since 
 $$p_B^{-1}(\prod_{\beta \in B} U_{\beta}) = \mathbb{G}'(\mathcal{O}_{K}) \cap \bigcap_{\beta \in B} p_{\beta}^{-1}\left( U_{\beta}\right)  = \mathbb{G}'(\PVS(K)),$$
we conclude that $\Lambda$ and $\mathbb{G}'(\PVS(K))$ are commensurable thanks to Corollary \ref{Corollary: Towers of star-approximate subsets are commensurable}.
  \end{proof}

 Finally, we deduce \emph{a posteriori} a superrigidity theorem without assumptions on the dimension of the target group. 
 
 \begin{proof}[Proof of Theorem \ref{Superrigidity}.]
  Let $T: \langle \Lambda \rangle \rightarrow \mathbb{L}(k)$. Take  $B$ finite, characteristic $0$ local fields $(k_{\beta})_{\beta \in B}$,  simple $k_{\beta}$-groups $(\mathbb{H}_{\beta})_{\beta\in B }$, a lattice $\Gamma \subset G_A \times \prod_{\beta \in B}\mathbb{H}_{\beta}(k_{\beta}) $  given by Theorem \ref{Main theorem}. Identify $\langle \Lambda \rangle$ with $\Gamma$. According to Margulis' superrigidity there is a continuous group homomorphism $\pi: G_A \times \prod_{\beta \in B}\mathbb{H}_{\beta}(k_{\beta}) \rightarrow \mathbb{L}(k)$ that extends $T$. Moreover, we know that $\pi$ factors though the natural projection to one of the simple factors of $G_A \times \prod_{\beta \in B}\mathbb{H}_{\beta}(k_{\beta})$. But $\pi(\L)=T(\L)$ is unbounded, so $\pi$ factors through the natural projection $p_{\alpha}:G_A \times \prod_{\beta \in B}\mathbb{H}_{\beta}(k_{\beta}) \rightarrow \mathbb{G}_{\alpha}(k_{\alpha})$ for some $\alpha \in A$. In particular, $\pi$ factors through the natural projection $p_A: G_A \times \prod_{\beta \in B}\mathbb{H}_{\beta}(k_{\beta}) \rightarrow G_A$. We thus have a continuous group homomorphism $G_A \rightarrow \mathbb{L}(k)$ that extends $T$.
 \end{proof}

 \section{Further Discussion} \label{Section: Further Discussion}
 \subsection{Quantitative information} 
 
 Suppose that $\Lambda$ is a $\star$-approximate lattice in the $\mathbb{R}$-points $G$ of a simple algebraic group as in Theorem \ref{First theorem}, and that $\Lambda$ is $l$-approximate for some integer $l$. Then, by Theorem \ref{First theorem}, there are a number field $K\subset \mathbb{R}$ and a $K$-group $\mathbb{H}$ such that $\mathbb{H}(\mathbb{R}) = G$  and $\Lambda$ is commensurable to $\mathbb{H}(\PVS(K))$.

 According to \cite[Lem. 2.3]{machado2019goodmodels}, the subset $\Lambda':=\Lambda^2 \cap \mathbb{H}(\mathcal{O}_K)$ is an $l^3$-approximate subgroup commensurable to $\mathbb{H}(\PVS(K))$. Let $v_0$ denote the place of $K$ coming from the inclusion $K \subset \mathbb{R}$. And write $\delta: \mathbb{H}(\mathcal{O}_K) \rightarrow \prod_{v} \mathbb{H}(K_v)$ the diagonal map where the product runs over all Archimedean places that are not equivalent to $v_0$ and such that $\mathbb{H}(K_v)$ is not compact. By weak approximation (e.g. \cite[II.6.8]{MR1090825}) we know that $\overline{\delta(\mathbb{H}(\PVS(K))}$ has non-empty interior, and, hence, $A:=\overline{\delta(\Lambda')}$ has non-empty interior too since finitely many left-translates of $A$ cover $\overline{\delta(\mathbb{H}(\PVS(K))}$. So $A$ is a compact $l^3$-approximate subgroup whose interior is not empty. 
 
 Let $H$ denote the Lie group $\prod_{v} \mathbb{H}(K_v)$, let $L$ be a maximal compact subgroup and choose a Haar measure $\mu$. Note that if $\Lambda$ is not commensurable to a lattice in $G$, then this product is not empty and $H$ has positive dimension. By \cite{jing2021nonabelian} we know that $\mu(A^2)^{1/\dim H/L} \geq 2 \mu(A)^{1/\dim H/L}$. But $A$ is an $l^3$-approximate subgroup, so we have $\mu(A^2) \leq l^3 \mu(A)$. Therefore, $\log l^3 \geq \dim H/L$. But we note that $\dim H \leq \left(\dim H/L\right) ^ 2$ (by the classification of simple Lie groups for instance). So $\log l^3 \geq \sqrt{\dim H}$ and $9\log^2 l \geq \dim H$.
 
 We note finally that $\dim H$ is bounded below by $$\dim \mathbb{G} \left(\vert \{ \text{Archimedean places $v$ of K s.t. } \mathbb{H}(K_v) \text{ non-compact }\}\vert -1 \right).$$
  So we have proved the first inequality. The second inequality then follows from \cite[I.3.2.1]{MR1090825} and \cite[Prop. 2.13.(iii)-(iv)]{bjorklund2016approximate}.

 \section{Acknowledgements}
 
 I am  indebted to my PhD supervisor, Emmanuel Breuillard, for invaluable discussions, support and encouragements. I am deeply grateful to the anonymous referee for comments and suggestions that helped improve this work immensely. This work was supported by the UK Engineering and Physical Sciences Research Council (EPSRC) grant EP/L016516/1 for the University of Cambridge Centre for Doctoral Training, the Cambridge Centre for Analysis. This material is based upon work supported by the National Science Foundation under Grant No. DMS-1926686.


\end{document}